\documentclass[draftcls,12pt, onecolumn]{IEEEtran}

\usepackage{amsmath}
\usepackage{amssymb}
\usepackage{amsfonts}
\usepackage{epsfig}
\usepackage{cite}
\usepackage{calc}
\usepackage{color}
\usepackage{psfrag,epstopdf}

\usepackage{times}
\usepackage{subfigure}
\DeclareGraphicsExtensions{.eps}

\newcommand{\ones}{{\mathbf 1}}
\newcommand{\norm}[1]{{||#1||}}

\newcommand{\R}{{\mathbb R}}
\newcommand{\RR}{{\mathcal R}}
\newcommand{\rv}{{\mathbf r}}
\newcommand{\xv}{{\mathbf x}}
\newcommand{\av}{{\mathbf a}}
\newcommand{\bv}{{\mathbf b}}

\newcommand{\dv}{{\mathbf d}}
\newcommand{\muv}{{\boldsymbol \mu}}
\newcommand{\Rvec}{{\mathbf R}}
\newcommand{\Wv}{{\mathbf W}}

\newcommand{\mcW}{\mathcal{W}}

\newcommand{\wv}{{\mathbf w}}
\newcommand{\A}{{\mathbb A}}

\newcommand{\XX}{{\mathcal X}}
\newcommand{\XXX}{{\mathbb X}}
\newcommand{\Y}{{\mathcal Y}}
\newcommand{\yv}{{\mathbf y}}

\newcommand{\D}{{\mathcal D}}
\newcommand{\K}{{\mathcal K}}

\newcommand{\N}{{\mathbb N}}

\newcommand{\Av}{{\mathbf A}}

\newcommand{\G}{{\mathcal G}}

\newcommand{\Ish}{{I^\sharp}}

\newcommand{\vect}[1]{{\boldsymbol{#1}}} 
\newcommand{\reqn}[1]{(\ref{#1})}
\newcommand{\tav}{\tilde{\mathbf a}}
\newcommand{\hav}{\hat{\mathbf a}}

\newcommand{\ba}{\begin{array}}
\newcommand{\ea}{\end{array}}

\newtheorem{defn}{Definition}
\newtheorem{thm}{Theorem}
\newtheorem{fact}{Fact}
\newtheorem{lem}{Lemma}
\newtheorem{prop}{Proposition}

\newtheorem{note}{Remark}

\newtheorem{assumption}{Assumption}

\newtheorem{property}{Property}
\newtheorem{obs}{Observation}

\newcommand{\nn}{\nonumber}

\newcommand{\beq}{\begin{equation}}
\newcommand{\eeq}{\end{equation}}
\newcommand{\bmu}{\begin{multline}}
\newcommand{\emu}{\end{multline}}
\newcommand{\bmun}{\begin{multline*}}
\newcommand{\emun}{\end{multline*}}

\newcommand{\bea}{\begin{eqnarray}}
\newcommand{\eea}{\end{eqnarray}}
\newcommand{\bean}{\begin{eqnarray*}}
\newcommand{\eean}{\end{eqnarray*}}

\newcommand{\bit}{\begin{itemize}}
\newcommand{\eit}{\end{itemize}}
\newcommand{\ben}{\begin{enumerate}}
\newcommand{\een}{\end{enumerate}}

\newcommand{\ignore}[1]{{}}

\DeclareMathOperator*{\letF}{F}
\DeclareMathOperator*{\letG}{G}

\definecolor{pink}{rgb}{1,.6,.6}

\title{Many-Sources Large Deviations for Max-Weight Scheduling}
\author{Vijay G. Subramanian, 
\thanks{Vijay Subramanian is with the Hamilton Institute, National University of Ireland, Maynooth, Co. Kildare, Ireland. He acknowledges SFI for the support of this research through grant 07/IN.1/I901.}
Tara Javidi,  
\thanks{Tara Javidi is with the Department of Electrical and Computer Engineering, University of 
California, San Diego, CA, USA.}
and Somsak Kittipiyakul
\thanks{Somsak Kittipiyakul is with the School of Information, Computer and Communication Technology, Sirindhorn International Institute of Technology, Thammasat University, Thailand. This work was carried out as a part of his Ph.D. dissertation while at the Department of Electrical and Computer Engineering, University of California, San Diego, CA, USA.}
}

\begin{document}

\thispagestyle{empty}
 \maketitle

\begin{abstract}
In this paper,  a many-sources large deviations principle (LDP) for the transient workload of a multi-queue single-server system is established where the service rates are chosen from a compact, convex and coordinate-convex rate region and where the service discipline is the max-weight policy. Under the assumption that the arrival processes satisfy a many-sources LDP, this is accomplished by employing Garcia's extended contraction principle that is applicable to quasi-continuous mappings. 

For the simplex rate-region, an LDP for the stationary workload is also
established under the additional requirements that the scheduling policy be work-conserving and that the arrival processes satisfy certain mixing conditions. 

The LDP results can be used to calculate asymptotic buffer overflow probabilities accounting for the multiplexing gain, when the arrival process is an average of \emph{i.i.d.} processes. The rate function for the stationary workload is expressed in term of the rate functions of the finite-horizon workloads when the arrival processes have \emph{i.i.d.} increments.  
\end{abstract}

\section{Introduction}

The drive to achieve maximum efficiency in wireless data networks and high-speed switches has lead to many advances in the design of good scheduling policies. One such family of good scheduling policies is an online policy\footnote{Online policies are those than can only use the past history of the arrivals, workloads and service decisions to decide on the scheduling choice; for example, these policies are not even aware of average arrival rates.} called the maximum weight (max-weight) scheduling policy. For the typical multi-class queue where only one queue can be served at a time, the max-weight policy serves one of the queues that has the largest value for the product of the workload 
and the service rate. We are interested in applying the max-weight policy to wireless networks where 
the scheduler is able to change the operating parameters at different levels of the traditional networking stack. Thus, 
the server has access to a richer choice of service options when compared to a traditional multi-class queue setting: each service option is a point within a compact, convex and coordinate-convex rate region. In this
setting, the max-weight policy naturally generalizes to finding an operating point within the rate region that has the maximum projection along the workload vector\footnote{In many ways cross-layer optimization has resulted in some firm strides towards a union of information theory and communication networks of the sort that was sought in \cite{hajekephremides1998}.}. Note that in this setting the traditional single-server multi-class queue has a rate-region given by a simplex. 

In the present article, with $K$ independent queues we seek to derive the probability of buffer overflow,
when the server scheduling follows a max-weight policy. More specifically, for a given finite value $B>0$,  we consider the two buffer overflow quantities. First,
we consider $P(\Wv_{0,T} \geq B \mathbf{1}_K)$ where $\Wv_{0,T} \in \R^K_+$ is the transient workload (to be formally defined later) at time $0$ with ``zero'' initial workload at time $-T$ and $\mathbf{1}_K\in\R^K_+$ is the vector of all $1$s. The second quantity we study is the stationary overflow probability for the limiting workload vector as $T \rightarrow \infty$, i.e., 
$P(\mcW \geq B \mathbf{1}_K)$.
 Since these probabilities are, in general, very hard to compute exactly, we consider logarithmic asymptotics to the probabilities of interest using the theory of large deviations. In particular, this paper, under a ``many-sources'' scaling regime, establishes  logarithmic asymptotics for 1) the transient
 workload for a compact, convex, and coordinate-convex rate region; and 2) the stationary workload for a simplex rate region using a work conserving scheduler.
 
 In the classical\footnote{The appellation ``classical" is taken from \cite{CruiseTalk2008} where a general scaling framework is presented that encompasses in a single setting all the different scalings used in the ``many-sources" scaling regime.} many-sources asymptotic, one considers a sequence of queueing systems indexed by the number of (independent) sources multiplexed (or averaged) over a particular queue, i.e., the arrival process to each queue is the average of $L$ processes. 
 The analysis focuses on the asymptotic behavior of the systems when $L \rightarrow \infty$.
 In our work, we consider a generalization of many sources asymptotic in 
which the input to the queueing system $L$ exhibits a sample path large deviations property (LDP) similar to 
that of the average of $L$ independent arrival streams (See Assumption~1).
Given a sample path large deviations principle (see Definition~\ref{defn:LDP}) for the arrival processes, we derive a large deviations principle for the workload under the max-weight scheduling polic. In particular, we first show that the finite-horizon workload is a quasi-continuous map of the arrival process, for both the regular version of the max-weight policy and for a work-conserving version of it. Then the first contribution of the paper is that the finite-horizon workloads satisfy and LDP. This is obtained using a recent extension of the contraction principle by J. Garcia \cite{Garci04}. Restricting our attention to the simplex rate region (corresponding to the traditional multi-class single server queue), we again use Garcia's extended contraction principle (along with a mixing condition assumption on the arrival process) to establish an LDP for the stationary workload. We should emphasize here that in contrast to related ``many-sources" LDP results on FCFS and Priority policies that can be shown to be continuous, our LDP is established for an inherently discontinuous map that results from the max-weight scheduling policy. 
The LDP results (Theorems~\ref{thm:FiniteHorizonLDP}, \ref{thm:FiniteHorizonLDP-NWC} and \ref{thm:InfiniteHorizonLDP}) directly imply that the probability of buffer overflow has an exponential tail whose decay rate is dictated by a good rate function determined by the statistics of the arrival process. This rate function can be expressed as a solution to a finite-dimensional optimization problem which has the same flavor of a deterministic optimal control problem. The final contribution of our work is to provide a simplified form for the corresponding rate functions,  when the arrival process has \emph{i.i.d.} increments. 
 
The outline of the paper is as follows.  In Section~\ref{sect:litSearch}, we briefly motivate and 
contextualize our
work in the larger body of literature on LDP analysis of queues as well as cross layer scheduling.
The problem formulation is given in Section~\ref{sect:ProblemFormulation}. 
Section~\ref{sect:Background} provides background and preliminary results.
The main results of the paper, which are the LDPs of the workloads, are given in Section~\ref{sec:overview} and proved in Section~\ref{sect:LDP_workloads}. Section~\ref{sect:RateFunction} gives simplified expressions 
of the rate functions. 
We conclude in Section~\ref{sect:Conclusion} with a discussion of future work. 
  
We close this section with a summary of various notation used in this work. We use bold letters to 
discriminate vectors from scalar quantities as well as their components. 
We denote the set of natural and non-negative real numbers by $\mathbb{N}$ and 
$\mathbb{R}^+$, respectively. We take $\otimes$ to represent the Kronecker product. 
For $0\le m_1 \le m_2$ integers 
and a vector sequence$(\Av_t, t\in\N)$ where $\Av_t\in \R_+^K$ for $K\in\N$, we define $\Av(m_1,m_2] := \sum^{m_2}_{t=m_1+1} \Av_t$ as the cumulative arrivals from $m_1+1$ until timeslot $m_2$ where addition applies coordinate-wise. For vector-valued sequence $\Av$ we write $\Av|_{(m_1,m_2]}$ to denote the finite subsequence  $\{\Av_{-m_2},\dotsc,\Av_{-m_1-1}\}$. 
For a vector $\xv\in \R_+^K$ and set $B \subset \R_+^K$, 
$\text{Proj}_{B}(\xv)$ denotes the projection of vector $\xv$ on the set $B$, $\text{int}(B):=\{\boldsymbol{\lambda}\in B : \exists \boldsymbol{\lambda}'\in B \text{ s.t. } \boldsymbol{\lambda} < \boldsymbol{\lambda}'\}$ denotes the set of points strictly inside $B$, 
 and $[\xv]^+ := \max\{\mathbf{0},\xv\}$ where the function applies coordinate-wise. 
Lastly, for any given function $\letF:\mathcal{X} \mapsto \mathcal{Y}$ on 
metric spaces $\mathcal{X}$ and $\mathcal{Y}$, and $x \in \mathcal{X}$, we use the notation 
\begin{align*}
\sideset{^x}{}\letF := \{y \in \mathcal{Y}: (\exists x_n \rightarrow x) \text{ such that } \letF(x_n) \rightarrow y\},
\end{align*} 
to denote the set of all cluster points (in $\mathcal{Y}$) of the images of sequences in $\mathcal{X}$ converging to a point $x\in\mathcal{X}$.
Note that   
$\sideset{^{(\cdot)}}{}\letF$, in general, is a correspondence (also called a set-valued function) from $\mathcal{X}$ to $\mathcal{P}(\mathcal{Y})$ (the power set of $\mathcal{Y}$).
However, 
$\sideset{^{(\cdot)}}{}\letF$ is single-valued at $x$, i.e., $\sideset{^x}{}\letF = \{\letF(x)\}$, if and only if $F$ is continuous at $x$.

\section{Related Work}\label{sect:litSearch}

In recent years, cross-layer scheduling has become a major focus of research in queueing and 
information theory due to its potential applications in communication networks. For brevity we do not list
many important works, and instead mention only those closely related to this work.

In our LDP analysis, we follow the lead of many recent papers on the analysis of scheduling algorithms  \cite{Berts98,StolSingleServer,Shakk08,Ying06,Subra08,SubraJournal08,%
Yang06,ShakkottaiManySource} by considering logarithmic asymptotics  to the probabilities of 
certain rare events. Maximum weight scheduling policy falls under class of the generalized $c\mu$-rule policies and is known to be stabilizing under very mild conditions~\cite{TassiulasEphremides92,mckeownswitch1999,daiprabhakar2000,armonybambos2003,Andrews}. A refined analysis of this policy shows that it minimizes the workload in the heavy traffic regime~\cite{vanmieghemGcmu1995,mandelbaumstolyar2004,stolyarMW_HT2004} over a large class of stationary online policies. This optimality of the max-weight policies also carries over to Large Deviations based tail asymptotes: the work-conserving version of these policies is known to minimize the exponent of the tail asymptote of the stationary workload over a large class of stationary, online and work-conserving policies~\cite{StolSingleServer}.

The present paper is closely related to  \cite{Subra08,SubraJournal08},
 where the buffer overflow probability for the workload processes of a single-server multi-queue queueing system under max-weight policies and general compact, convex and coordinate-convex capacity regions was established. While \cite{Subra08,SubraJournal08} addresses the ``large-buffer" scaling regime, this paper establishes similar logarithmic asymptotics results under the ``many-sources'' scaling regime (see \cite{Weiss86,Botvi95,Courc96,Simon95,ShwartzWeiss1995,LikhanovMazumdar1999,DebickiMandjes2003,WeirmanJournal2008,KotopoulosMazumdar2008,DelasMazumdarRosenberg2002,OzturkMazumdarLikhanov2004,MandjesBorst2000,MandjesUitert2005,Wischik1999,Wischik01,Wisch01,MandjesKim2001,Ganes04,Yang06,ShakkottaiManySource,YangShakkottai2006}). As the body of work on the ``many-sources" scaling regime has grown, results have been established for many different scheduling policies for a single-server queue and also for networks of queues, namely, FCFS~\cite{BuffetDuffield94,Duffield94,Botvi95,Simon95,Courc96,
ShwartzWeiss1995,LikhanovMazumdar1999,Wisch01,Ganes04,DebickiMandjes2003,MandjesKim2001,MandjesBorst2000,MandjesUitert2005}, priority queueing~\cite{Wisch01,Ganes04,DelasMazumdarRosenberg2002,MandjesUitert2005,ShakkottaiManySource}, GPS~\cite{KotopoulosMazumdar2008}, and SRPT and similar policies~\cite{YangShakkottai2006,Yang06,WeirmanJournal2008}. Our LDP analysis of max-weight scheduling is strongly motivated and complemented by these papers. 

Finally, we close this section with a discussion of our motivation to consider the ``many-sources" 
scaling studied in this apper.  The interest in the ``many-sources" asymptotic is known to be best motivated 
by 1) a recent and practical interest in applications when there are large number of flows to each user or node. This asymptote usually gives a more refined approximation to the probabilistic quantities of interest by incorporating the impact of the multiplexing gain~\cite{Weiss86,ChoudhuryLucantoniWhitt93,BuffetDuffield94,Duffield94,Botvi95,Simon95,Courc96,
ShwartzWeiss1995,LikhanovMazumdar1999,DebickiMandjes2003,MandjesBorst2000,WeirmanJournal2008,KotopoulosMazumdar2008,DelasMazumdarRosenberg2002,OzturkMazumdarLikhanov2004,Wischik1999,Wischik01,Wisch01,MandjesKim2001,MandjesUitert2005,Ganes04,Yang06,ShakkottaiManySource,YangShakkottai2006} obtained by averaging many traffic sources together. However, our interest
in the ``many-sources" asymptotic has also been fueled by our earlier work \cite{KittiHighSNR07} on  2) a cross-layer optimization of the PHY layer parameters, e.g.,  
duration of the finite code blocks or cooperative cluster size when the fading channel is operated at high signal-to-noise-ratio (SNR). The high SNR regime is a very natural setting for the many-sources scaling since the capacity of the channel typically scales to infinity as $\log(\mathrm{SNR})$, and therefore it is natural to scale the arrival rate of the flows with the same parameter, which is best accomplished by multiplexing more sources; in other words by setting $L\propto \log(\mathrm{SNR})$. 
The present work on the many sources large deviation analysis of max-weight 
provides a first step in extending the above cross-layer optimization to a multi-user setting,  
an important topic for future research.


\section{Problem Formulation} \label{sect:ProblemFormulation}
We consider a discrete-time queueing system with $K\in \N$ independent queues and one server. 
We are interested in the statistical properties of the unfinished workload in queue $k$ at time $t$ under
a max-weight server allocation policy. 
Let $W^k_t \in \R_+$ be the unfinished workload (queue length) of queue $k$ at the beginning of time $-t$ and $R^k_t$ be the amount of service allocated to  queue $k$ during time $(-t,-t+1]$.  
Let $\Wv_t := (W^k_t,k\in \K)$ be the corresponding workload vector and  $\Rvec_t := (R^k_t,k\in \K)$ be the rate vector.  For every queue $k\in\K:= \{1,\ldots,K\}$ we assume that work (in bits) arrives into the queue given by a sequence $(A^k_t, t\in\N)$ where $A^k_t\in \R_+$ is the work brought in at time $-t$. 
For $t\in \N$, the dynamics of the workloads of queue $k\in \K$ is \beq W^k_{t-1} = [W^k_{t}  - R^{k}_{t}]^+ + A^k_{t}. \label{eq:qdynamic}\eeq Note that we assume that the arrivals $\Av_t$ happen any time in $(-t,-t+1)$ but cannot be served in that timeslot $-t$. 

The set of server's operating points is restricted to a compact and convex set 
$\RR\subset {\mathbb{R}}_+^K$ known as the \emph{capacity} or {\emph{rate region}} of the server, i.e., $\Rvec_t \in \RR$, for all time $t$. We make the simplifying assumption that bits are infinitely divisible so that the rate allocations can be assumed to be real numbers. Furthermore, we also assume that $\RR$ is coordinate-convex, i.e, if $\boldsymbol{\mu_1}\in\RR$, then every $\boldsymbol{\mu_2}\in {\mathbb{R}}_+^K$ such that $\boldsymbol{\mu_2} \leq \boldsymbol{\mu_1}$ is also in $\RR$ where the inequalities apply along each coordinate. We are 
interested in the max-weight  scheduler, 
and its closely related work-conserving version. 
At the beginning of timeslot $-t$, the rate vector $\Rvec_t \in \RR$ 
is selected by a max-weight scheduler 
in response to the current workload $\Wv_t$. Specifically, under
max-weight scheduler 
and in response to the current workload $\Wv_t$, 
the rate vector $\Rvec^*_t$ 
is chosen 
such that 
\beq \label{eq:maxwt}
\Rvec^{*}_t \in  \arg\max_{\Rvec \in \RR} <\Rvec , \Wv_t>. 
\eeq
As later established by Lemma~\ref{lem:cont_H}, it is possible to construct a quasi-continuous (see Defn~\ref{defn:quasi-continuity_condition} in Section~\ref{sect:Background}) function $H$ such that
$\Rvec^{*}_t =H(\Wv_t)$. We call this construction the (max-weight) scheduling function. 
We also define a 
non-idling modification of max-weight scheduler for which the rate of service $\Rvec_t^{**} $ 
is such that it splits the service when the unfinished workload in each queue $k$ is less than 
$C^k =\max \{R^k: \Rvec \in \RR\} $. 
Lemma~\ref{lem:cont_H} also shows that it is possible to construct a 
quasi-continuous 
function $H^{\mathrm{wc}}$ such that
$\Rvec^{**}_t =H^{\mathrm{wc}}(\Wv_t)$ and 
\beq \label{eq:Hwc}
H^{\mathrm{wc}}(\Wv_t) = 
\begin{cases}
 \text{Proj}_{\RR}(\Wv_t) &  \text{if }\Wv_t \in \Pi_{k=1}^{K} [0,C_k); \\
H(\Wv_t) &  \text{otherwise}.
\end{cases}
\eeq
When necessary we will distinguish the workload vectors that result from the work-conserving max-weight by labeling them as $\Wv^{\mathrm{wc}}$. 

As mentioned earlier, in this paper we are interested in the probability distributions for
the \emph{finite-horizon} and \emph{infinite-horizon} workloads. The finite-horizon workload, denoted by $\Wv_{0,T}$, is the workload at time $0$, assuming the initial condition at time $-T$ is $\Wv_T =\mathbf{0}$ . The index $T$ in $\Wv_{0,T}$ reminds us of this initial condition.\footnote{Note that the result remains valid even when the initial condition is within $\RR$ with the work-conserving scheduler. With $\Wv_T \in \RR$, we always have the workload at time $-T+1$ be $\Wv_{T-1} = [\Wv_T - H^{\mathrm{wc}}(\Wv_T)]^+ + \Av_T = \Av_T$ from the non-idling condition that we imposed on the server allocation mechanism as $\text{Proj}_{\RR}(\Wv_T)=\Wv_T$.}  
The infinite-horizon workload, $\mcW$, is defined as $ \mcW  := \lim_{T \rightarrow \infty} \Wv_{0,T}$. 
We  assume that the limit exists but may be infinite. 
Note that our results for the infinite-horizon workload are obtained in the restricted setting 
of a work conserving max-weight scheduler operating on a simplex rate region. 
For this work-conserving max-weight scheduler, it is known that $\mcW^{\mathrm{wc}}$ is the stationary workload when the system is stable.

We will use functions $G_T$ and $G^{\mathrm{wc}}_T$ 
to relate the arrival process $\Av|_{(0,T]}$ with the unfinished work under max-weight scheduling, $\Wv_{0,T}$ and under work conserving max-weight $\Wv_{0,T}^{\mathrm{wc}}$, i.e. 
$\Wv_{0,T} = G_T (\Av|_{(0,T]})$ and $\Wv^{\mathrm{wc}}_{0,T} = G^{\mathrm{wc}}_T (\Av|_{(0,T]})$. Similarly, we also define function $G^{\mathrm{wc}}$ to describe the workload under
work-conserving max-weight when the arrival sequence is given,  
i.e.\ $\mcW^{\mathrm{wc}} = G^{\mathrm{wc}}(\Av)$; we do not indicate the rate-region here as it is implicitly understood to be the simplex rate-region.

For each user $k \in \K $ and system indexed by $L\in\N$, we will assume a stationary arrival process of work brought into the system given by a sequence $A^{k,L}:=(A^{k,L}_t, t\in \N)$ where $A^{k,L}_t \in \R_+$ is the work (in bits) brought in at time $-t$ into the queue of user $k$. The arrivals to different queues/users are assumed to be mutually independent. Also let $\Av^L := (A^{k,L},k\in \K)$ be the sequence of arrival vectors. In our large deviation analysis, we characterize the 
asymptotic probability distributions for the finite-horizon and infinite-horizon workloads,  
$G_T(\Av^L|_{(0,T]})$ and $G^{\mathrm{wc}}_T(\Av^L|_{(0,T]})$,
as $L \rightarrow \infty$. 


We close this section by noting that in its typical avatar a max-weight scheduler is also accompanied by a set of non-zero weights $\boldsymbol{\beta}\in \mathbb{R}_+^K$ to weigh the workload vectors while determining the max-weight service vector. Using the observations in \cite{Subra08,SubraJournal08} it can be seen that there is no loss of generality in assuming that the weight vector is $\mathbf{1}_K$. 

\section{Background and Preliminaries} \label{sect:Background}

In this section, we provide a brief review of fundamental definitions, concepts, and relevant 
results in large deviations theory that are essential for understanding our paper. 
\ignore{Since we use these notions and definitions in our modeling, as well 
as analysis, the following section is essential in facilitating the understanding of this paper, 
particularly by readers not familiar with the topics. Readers knowledgable about the topic can 
skip this section.} Except for Lemma~\ref{lem:cont_H}, which establishes the 
quasi-continuity of the scheduling functions $H$ and $H^{\rm{wc}}$, 
the material in this section can be found 
in \cite{Wisch01,Ganes04,Garci04,Dembo98}.

\subsection{Quasi-continuity and Almost Compactness}

In this section, we recall two important analytic properties for functions on metric spaces: 
quasi-continuity and almost compactness. These properties allow for an 
extension of contraction principle to which we later appeal. 
First, let us provide the definition of quasi-continuity on metric spaces:


\begin{defn}\cite[Theorem 3.2]{Garci04} \label{defn:quasi-continuity_condition} 
Let $\XX,\Y$ be complete metric spaces. Function $F:\XX \mapsto \Y$ 
is {\emph{quasi-continuous}} at $x\in \XX$ if and only if 
for each $x\in \XX$, there is a sequence $\{x_n\}$ such that $x_n\rightarrow x, F(x_n) \rightarrow F(x)$,
and such that for all $n,$ $F$ is continuous at $x_n$.
\end{defn}

\begin{note} 
Obviously, every continuous function is quasi-continuous. A step function $F:\R \mapsto \R$, where $F(x) = 0$ for $x<0$, $F(x) = 1$ for $x\ge0$, is quasi-continuous. However, if $F(0) = 1/2$, then $F$ is not quasi-continuous. 
\end{note}

\begin{note} \label{Rem:CompositionQuasi}
An important property that we will user later is that if $F$ is a continuous function and $G$ is a quasi-continuous function, then $F \circ G$ is quasi-continuous. However, $G \circ F$ is not necessarily quasi-continuous \cite{Garci04}. 
\end{note}

As stated before, max-weight and its work conserving variation allow for  
quasi-continuous function selections.
\begin{lem} \label{lem:cont_H}
There exist quasi-continuous functions 
$H$ and $H^{\mathrm{wc}}$ such that:
\beq
H(\Wv_t) \in  \arg\max_{\Rvec \in \RR} <\Rvec , \Wv_t>,  \nonumber
\eeq
and 
\beq 
H^{\mathrm{wc}}(\Wv_t) = \left\{ \begin{array}{ll}
 \text{Proj}_{\RR}(\Wv_t) &  \Wv_t \in \Pi_{k=1}^{K} [0,C_k) \\
H(\Wv_t) &  \mbox{otherwise} \end{array}. \right.  \nonumber
\eeq
\end{lem}
\begin{IEEEproof}
See Appendix~\ref{sec:app1}. The proof relies on the structural properties of the scheduling maps.
\end{IEEEproof}

Having provided the definition of quasi-continuity, 
we are ready to define almost compactness for a function:
\begin{defn}\cite[Lemma 6.1]{Garci04} \label{defn:almost-compact_condition} 
If $\XX,\Y$ are complete metric spaces, a function $F:\XX \mapsto \Y$ is {\emph{almost compact}} at $x \in \XX$ 
if for every sequence ${x_n}$ converging to $x$, there is a subsequence along which $F$ converges to a point $y \in \Y$.
We say that $F$ is almost compact if it is almost compact at every $x \in \XX$. 
\end{defn}

\subsection{Topology for Sample Paths}\label{sect:topology}
Since a large deviations principle is defined using topological entities and since we will deal with continuity and convergence of the workload mappings, we need to precisely specify the topology for the space of the arrival sample paths. We use the scaled-uniform norm/weighted supremum norm 
topology~\footnote{In the theory of weak convergence of probability measures this topology was first proposed in \cite{Muller1968,BorovkovSahanenko1973} for continuous functions vanishing at infinity with a continuous index set. In the same context it was then generalized to \emph{cad-lag} functions with a continuous index set in \cite{Whitt1972}. The most general setting of this topology for discrete-time processes can be found in \cite{Borovkov1974,Sahanenko1974}, and the corresponding setting for continuous-time processes can be found in \cite{Bauer1981}. The usage of this topology in the context of Large Deviations can be found in \cite{DeuschelStroock1989} and \cite{GaneshConnell2002}. Finally, the central theme in \cite{GaneshConnell2002,Ganes04,Wisch01} is to demonstrate how this is a natural topology to use in the queueing context.} used in \cite{Wisch01} for our analysis.

Let $\D\subset \{x:\N \mapsto \R_+\}$ denote the space of non-negative sequences such that $\sup_{t\in \N} \left|\frac{x(0,t]}{t}\right| <+\infty$, and let $\D^K$ be the $K$ cartesian product of $\D$. Let $||\cdot||_u$ be the scaled uniform norm on $\D$, i.e., $ ||x||_u := \sup_{t\in \N} \left|\frac{x(0,t]}{t}\right|$ for all $x\in \D$, while for all $a = (a^k, k\in \K) \in \D^K$, where $a^k \in \D$, the scaled uniform norm of $a$ is 
$ ||a||_u := \max_{k\in\K} ||a^k||_u$. The space $\D$ is metrizable via the scaled uniform norm $||\cdot||_u$, i.e., for all $x,y\in\D$, the distance between them is $||x-y||_u=\sup_{t\in \N} \left| \frac{x(0,t]-y(0,t]}{t} \right|$. 
Define a subspace $\D_\mu$ of $\D$ which contains all the arrival paths whose average arrival rate is equal to the expected rate $\mu$, i.e., $\D_{\mu} := \left\{x\in \D: \lim_{t\rightarrow \infty} \frac{x(0,t]}{t} = \mu \right\}$. Also for a vector $\muv =(\mu_1, \ldots, \mu_K)$ 
define $\D^K_\muv$ to be the product space of $\D_{\mu^k}$ for all $k\in\K$. We equip $\D_\mu$ and $\D^K_\muv$ with the appropriate subspace and product topologies \cite{Munkr00}. 
For finite dimensional metric spaces like $\R^n_+$, $n\in \N$, we use the square uniform topology (sames as the product topology) with the square metric $\rho$ \cite{Munkr00}, where $\rho(\xv,\yv) := \max_{i\in \{1,\ldots,n\}} |x^i-y^i|$. From \cite{Wisch01,Ganes04} it is also clear that the scaled uniform norm topology is stronger than the point-wise convergence topology, hence the projection and shift operators are continuous under the scaled uniform norm topology.

\subsection{Large Deviations Principle} 
The following definition of a large deviations principle is taken from \cite{Wisch01}. For a complete reference to the theory, definitions, and tools, see \cite{Dembo98}.

\begin{defn}[Large deviations principle] \label{defn:LDP} 
A sequence of random variables $X^L$ in a Hausdorff space $\mathcal X$ with $\sigma$-algebra $\mathcal B$ is said to satisfy a large deviations principle (LDP)\footnote{Often $\mathcal{B}$ is taken to be the Borel $\sigma$-algebra, and a rate function is, by definition, non-negative and lower semicontinuous.} with good rate function $I$ if, for any $B\in \mathcal B$, 
\begin{align} 
\begin{split}
  -\inf_{x\in B^o} I(x) & \le \liminf_{L\rightarrow \infty} \frac{1}{L} \log P(X^L \in B)
\le \limsup_{L\rightarrow \infty} \frac{1}{L} \log P(X^L \in B) \le -\inf_{x\in \bar{B}} I(x),
\end{split}
\end{align} 
where $B^o$ and $\bar B$ are the interior and the closure of $B$, respectively, and if the rate function $I:\mathcal{X} \mapsto \R_+ \cup \{\infty\}$ has compact level sets, where the level sets are defined as $\{x: I(x) \le \alpha\}$, for $\alpha \in \R$.

If $X^L$ is a mapping from $\N$ to $\R$ describing the sample path of a random sequence, the LDP is referred to as a \emph{sample path} LDP. 
\end{defn}


\subsection{Garcia's Extended Contraction Principle}
The contraction principle (see \cite[p. 126]{Dembo98}) says that if we have an LDP for a sequence of random variables, we can effortlessly obtain LDP's for a whole other class of random sequences that are obtained via continuous transformations. 
However, due to the inherent discontinuity in the max-weight scheduling function, this (regular)
contraction principle fails to provide sufficient structure in the setting of our interest. 
Instead, we will utilize the following powerful extension of the contraction principle for quasi-continuous transformations on metric spaces, given by Garcia \cite{Garci04}. 
Garcia's extended contraction principle~\cite[Theorem~1.1]{Garci04} then says the following:
\begin{fact}
\label{thm:Garcia} 
Assume that $\Omega \stackrel{X^L}{\rightarrow} \XX \stackrel{F}{\rightarrow} \Y$, $\XX, \Y$ are metric spaces, and $\{X^L\}$ satisfies an LDP in $\XX$ with rate function $\Ish$. If at every $x$ with $\Ish(x) < \infty$, the following hold:
\begin{enumerate}
\item $F$ is almost compact;  and
\item for all $y\in\sideset{^x}{}\letF$, there exists a sequence $\{x_n\}$ converging to $x$ such that 
$F(x_n) \rightarrow y$, $F$ is continuous at $x_n$, and $\Ish(x_n) \rightarrow \Ish(x)$,
\end{enumerate}
then $\{F(X^L)\}$ satisfies an LDP with rate function given by 
\beq I(y) = \inf \left\{\Ish(x): y \in \sideset{^x}{}\letF  \right\}.\eeq
\end{fact}
\begin{note}
Whenever $\Ish(\cdot)$ is continuous, then the second condition is reduced to confirming 
that $F$ is a quasi-continuous function.
\end{note}

\section{Assumptions and Overview of the Results}\label{sec:overview}

Garcia's extended contraction principle together with Lemma~\ref{lem:cont_H} suggests 
the following road map to obtaining an LDP for the finite and infinite horizon workload processes under
max-weight scheduling. 
The large deviation property for the sequences of finite- and infinite-horizon workloads would follow as a direct consequence of the sample-path LDP of the arrival process, as soon as one establishes
the quasi-continuity and almost compactness of the mappings $G_t$, $G^{\mathrm{wc}}_t$, and
$G^{\mathrm{wc}}$ along with some continuity properties of the rate function obtained from the sample-path LDP assumption on the arrival processes. 
As a result, our first task is to restrict our attention to arrival streams that satisfy the
sample-path LDP as stated by Assumptions~\ref{Assumption1}-\ref{Assumption2} 
in Section~\ref{assumptions}. There we also discuss a family of 
arrival processes  which satisfies these assumptions.


\subsection{Sample Path LDP of Arrival Processes}\label{assumptions}
The following sample path LDP for the sequence of arrival processes $\Av^L$ is the starting point of our analysis.


\begin{assumption}[Many-sources sample path LDP] \label{Assumption1}
The sequence $\{\Av^L\}$ satisfies a sample path LDP in $\D^K_{\boldsymbol{\mu}}$ equipped with the scaled uniform topology with rate function $\Ish$, where 
the rate function $\Ish$ is given as \beq \label{eq:defn_I} \Ish(\av) := \sup_{t \in \N} \Ish_t(\av|_{(0,t]}) = \lim_{t\rightarrow \infty} \Ish_t(\av|_{(0,t]})\eeq for $\av\in \D^K_{\boldsymbol{\mu}}$, where for every $t\in\N$ we also assume that $\{\Av^L|_{(0,t]}\}$ satisfies an LDP with rate function $\Ish_t(\cdot)$ (in the product topology). 
\end{assumption}

\begin{note} 
The most general conditions for Assumption~\ref{Assumption1} to be satisfied are given in \cite[Thm. 3]{Wisch01} (also stated in \cite[Thm. 7.1, pg. 156]{Ganes04}). There it is also shown how several standard stationary processes used for traffic modeling, such as \emph{i.i.d.} increment processes, Markov-modulated, a general class of Gaussian, and fractional Brownian processes (for long-range dependent or heavy-tailed traffic), satisfy Assumption~\ref{Assumption1}. The conditions of  \cite[Thm. 3]{Wisch01} also imply that the sequence $\{A^L\}$ also satisfies an LDP on $\D^K$ equipped with the scaled uniform topology, with rate function $\Ish$ where $\Ish(a) = \infty$ for $a \in \D^K / \D^K_{\boldsymbol{\mu}}$ \cite{Wisch01,Ganes04}. Finally, it is shown in 
\cite[Lemma 7.8]{Ganes04} that under the conditions of  \cite[Thm. 3]{Wisch01},  for all $t\in\N$ we have $\Ish_t(\muv\otimes\ones_{t}) = 0$ and also that $\Ish_t(\cdot)$ is convex. 
\end{note}

In this paper, we further assume the following continuity conditions on the rate functions: 
\begin{assumption} \label{Assumption-cont} 
We assume that $\Ish_t(\cdot)$ is continuous in the product topology on $\Re_+^{K\times t}$.
\end{assumption}

\begin{assumption} \label{Assumption2}
For every point $x$ in the effective domain of $\Ish$, i.e., $\{\av \in \D_{\boldsymbol{\mu}}^K: \Ish(\av) < +\infty\}$, we assume that for every sequence $\{\av^n\}$ converging to $\av$ in $\D^K_{\boldsymbol{\mu}}$ such that there exists $t\in \N$ so that $\av^n_s=\av_s$ for all $s > t$ (for all $n$), we have $\Ish(\av^n)\rightarrow \Ish(\av)$.  
\end{assumption}

\begin{note}\label{remcont}
Assumption~\ref{Assumption2} is a manifestation of a mixing condition that provides a certain independence of the long-term behaviour of the process with respect to any finite initial block/window.\end{note}

In Proposition~\ref{prop1} below we demonstrate that processes with \emph{i.i.d.} increments naturally satisfy Assumptions~\ref{Assumption1}-\ref{Assumption2}. Before proving this we define a coercive function as follows. 
\begin{defn}
A function $f:\R_+\mapsto\R_+\cup\{+\infty\}$ with domain $\mathrm{Dom}(f):=\{x: f(x) <+\infty\}$ is defined to be \emph{coercive} if for every $y\in\R_+$, there is compact set $\mathcal{C}\subsetneq \mathrm{Dom}(f)$ such that $f(x) > y$ for all $x\in \mathrm{Dom}(f)\setminus \mathcal{C}$.
\end{defn}


\begin{prop}\label{prop1}
For scale parameter $L$ assume that the arrival process has \emph{i.i.d.} increments while the arrival processes for different users are independent. Furthermore, assume that for every $k$, 
$\{A_1^{k,L}\}$ satisfies the conditions of the G\"artner-Ellis Theorem~\cite[Thm. 2.3.6]{Dembo98} and that the limiting rate function is either coercive or has domain $\Re_+$. 
Then Assumptions~\ref{Assumption1}-\ref{Assumption2} hold. 
\end{prop}
\begin{proof}
The fact the above class of processes satisfy Assumption~\ref{Assumption1} is immediate from \cite[Thm. 3]{Wisch01}. Similarly the G\"artner-Ellis Theorem~\cite[Thm. 2.3.6]{Dembo98} yields convexity and from the coercivity of the rate function or with its domain being $\R_+$, Assumption~\ref{Assumption-cont} follows. When the arrival process has \emph{i.i.d.} increments, then it also follows that for 
$\yv \in \mathbb{R}_+^{k\times t}$ and  $\xv \in \D^{K}$, 
\begin{align*}
\Ish_t(\yv)=\sum_{k=1}^K \sum_{s=1}^{t} \Lambda_1^{*,k}(y^k_s),
\end{align*}
and
\begin{align*}
\Ish(\xv)=\sum_{k=1}^K \sum_{s=1}^{+\infty} \Lambda_1^{*,k}(x^k_s),
\end{align*}
where $\Lambda_1^{*,k}$ is the Fenchel-Legendre transform of $\Lambda_1^k(\theta):=\lim_{L\rightarrow\infty}\frac{1}{L} \log Ee^{\theta L A_1^{k,L}}$.
Note that that $\Ish(\xv)=+\infty$ if $\xv\in \D^K \setminus \D^K_{\muv}$. 

Now for every sequence $\{\xv^n\}$ converging to $\xv$ in $\D^K_{\boldsymbol{\mu}}$ with $\Ish(\xv)<+\infty$ such that there exists $t\in \N$ so that $\xv^n_s=\xv_s$ for all $s > t$ (for all $n$), we have $\Ish(\xv^n)<+\infty$ for all $n$ large enough and 
\begin{align*}
\Ish(\xv^n) - \Ish(\xv) =\Ish_t(\xv^n|_{(0,t]}) - \Ish_t(\xv|_{(0,t]}) \xrightarrow[n\rightarrow+\infty]{} 0.
\end{align*}
\end{proof}
\begin{note}
A more general characterization, especially of Assumption~\ref{Assumption2}, remains an important area of future work. However, 
to provide insight about Assumptions~\ref{Assumption-cont}-\ref{Assumption2}
and their relationship to Assumption~\ref{Assumption1} we provide the following simple 
example. Our example will be constructed as averages of $L$ stationary \emph{i.i.d.} random processes so it is sufficient to describe the underlying stochastic process: We let the random process be $X_k = \gamma + \gamma V$ when $k$ is odd and $X_k = \gamma - \gamma V$ when $k$ is even, where
$V$ is a $\mathrm{Uniform}[-1,1]$ random variable and $\gamma > 0$ is the long-term mean for all our sequences. Here it can be verified that $I(x)<+\infty$ if and only if $x_k=\gamma(1-\alpha)$ if $k$ is even and $x_k=\gamma(1+\alpha)$ for $k$ odd where $\alpha\in[-1,1]$. It can also be verified that $I(x)=0$ for the all-$\gamma$ sequence, i.e., $x_k\equiv \gamma$. However, 
it is clear that one can easily construct sequences as required in 
Assumption~\ref{Assumption2} with the rate function being infinite which, 
nevertheless, converge to the all-$\gamma$ sequence. In other words, 
Assumptions~\ref{Assumption-cont}-\ref{Assumption2} require more than 
the regularity conditions from \cite{Wisch01,Ganes04} that guarantee 
validity of Assumption~\ref{Assumption1}. 
\end{note}

Given the above assumptions on the arrival processes, we are now ready to 
provide an overview of the main results
of the paper.

\subsection{Main Results: An Overview}

Assuming that the sequence of the arrival processes $\{\Av^L\}$ satisfies a 
many-sources sample-path LDP with a ``continuous" rate function 
(Assumptions~\ref{Assumption1}-\ref{Assumption2}), 
LDPs for the finite and infinite-horizon workloads will be a direct consequence of 
Garcia's extended contraction principle once the required quasi-continuity and almost
compactness properties are shown. We will demonstrate that the quasi-continuity and almost compactness of the finite horizon workload mappings are inherited from the quasi-continuity of the schedulers
$H$ and $H^{\mathrm{wc}}$.

\begin{thm} \label{thm:FiniteHorizonLDP}
Under Assumptions~\ref{Assumption1}-\ref{Assumption-cont} and for all $t\in \N$,
 the sequence of the finite-horizon workloads under a max-weight 
scheduler $\{\Wv_{0,t}^L:= G_t(\Av^L|_{(0,t]})\}$
satisfies an LDP on $\R^K_+$ with the rate function $I_t$, where for $\bv \in \R^K_+$ \beq I_t(\bv) = \inf_{\xv \in \R^{K\times t}_+: \sideset{^\xv}{_t}\letG \ni \bv}  \Ish_t(\xv).\label{eq:defnIt} \eeq 
\end{thm}

\begin{thm} \label{thm:FiniteHorizonLDP-NWC}
Under Assumptions~\ref{Assumption1}-\ref{Assumption-cont} and for all $t\in \N$,
 the sequence of the finite-horizon workloads under the work-conserving max-weight 
scheduler, 
$\{\Wv_{0,t}^L:= G^{\mathrm{wc}}_t(\Av^L|_{(0,t]})\}$, satisfies an LDP on $\R^K_+$ with the rate function $I_t$, where for $\bv \in \R^K_+$ \beq I_t(\bv) = \inf_{\xv \in \R^{K\times t}_+: \sideset{^\xv}{^{\mathrm{wc}}_t}\letG \ni \bv}  \Ish_t(\xv).\label{eq:defnItgen} \eeq 
\end{thm}

The quasi-continuity and almost compactness of the stationary workload mapping, however, 
requires slightly more work as shown in Section~\ref{sect:LDP_workloads}.  
In fact, unlike the finite-horizon LDP results of Theorems~\ref{thm:FiniteHorizonLDP} and \ref{thm:FiniteHorizonLDP-NWC}, which hold for general rate regions 
and for both max-weight 
and its work conserving version, the infinite horizon 
LDP result of Theorem~\ref{thm:InfiniteHorizonLDP} is established only under the 
work conserving max-weight policy and only with a simplex rate region. For a vector $\xv\in\R^K_+$ 
define $\hat{\xv}:=\sum_{k=1}^K x^k/C^k$, then the simplex rate region is given by,
\beq 
\RR_s:= \left\{\rv \in \R^K_+: \hat{\rv} \le 1 \right\}.\label{simplex} 
\eeq

\begin{thm} \label{thm:InfiniteHorizonLDP}
Consider a work conserving max-weight scheduler with a simplex rate region $\RR_s$.
Let $\vect{\mu} \in \mathbb{R}^K$ be an admissible arrival rate vector strictly inside
this simplex rate region, i.e., 
$\vect{\mu} \in \mbox{int}(\RR_s)$ 
and $\{\Av^L \}$ be a sequence of 
arrival processes that satisfies  Assumptions~\ref{Assumption1}-\ref{Assumption2}
in $\D^K_{\vect{\mu}}$.  
The sequence of infinite-horizon workloads $\{\mcW^L := G^{\mathrm{wc}}(\Av^L)\}$ satisfies an LDP on $\R^K_+$ with good rate function $J$, where for $\bv \in \R^K_+$ \beq J(\bv) = \inf_{\av\in \D^K_{\vect{\mu}}: \sideset{^\av}{^{\mathrm{wc}}}\letG \ni \bv} \Ish(\av). \label{defn:J} \eeq
\end{thm}

Note that all the rate functions have the appearance of being what one would naturally expect, i.e., among all arrival sequences that result in the required workload at the required epoch, find the one with the least cost to deduce the rate function. However, the discontinuity of the queueing map makes this a non-trivial assertion and also enlarges the set of allowed arrival sequences.

\section{Analysis: LDP's for Workloads} \label{sect:LDP_workloads}

In this section, we prove the main results of the paper: establishing LDP for the sequences of the finite-horizon and infinite-horizon workloads. We first delineate the proof for the LDPs for the sequence of the finite-horizon workloads. 

\subsection{LDP for Finite-Horizon Workloads}

In this section, for $t\in \N$, we establish an LDP for finite-horizon workloads 
$\{\Wv^L_{0,t} := G^{\mathrm{wc}}_t(\Av^L|_{(0,t]})\}$ and $\{\Wv^L_{0,t} := G_t(\Av^L|_{(0,t]})\}$. The approach is to first show that the mappings $G^\mathrm{wc}_t: \R^{K t}_+ \mapsto \R^K_+$ and 
$G_t: \R^{K t}_+ \mapsto \R^K_+$ are quasi-continuous and almost compact and 
use Garcia's extended contraction principle to obtain an LDP for the finite-horizon workloads from the LDP assumption for $\{\Av^L|_{(0,t]}\}$. 
From Fact~\ref{defn:quasi-continuity_condition}, the almost compactness and quasi-continuity of workload mappings are sufficient to establish an LDP for finite-horizon workload, since according to Assumption~\ref{Assumption-cont}, $\Ish_t(\cdot)$ is continuous (in the product topology).   

Using the quasi-continuity of the scheduling function we now prove the required quasi-continuity of the workload maps. First we consider the work-conserving max-weight scheduler. 
\begin{lem} \label{lem:continuity_Gt}
For $t\in \N$, $G^{\mathrm{wc}}_t$ is quasi-continuous with respect to the scaled uniform topology.
\end{lem}
\begin{IEEEproof} See Appendix~\ref{sec:app2}. The induction-based proof uses the quasi-continuity of the scheduler $H$ and the linear dependence of the workload $\Wv_s$ at time $-s$ on $\Av_{s+1}$ for all $s\in (0,t-1]$. 
\end{IEEEproof}

Similarly, in order to establish an LDP for the workload process under 
the maximum weight scheduler, 
we have the quasi-continuiy of function $G_t$.
\begin{lem} \label{lem:quasic_mw}
For $t\in\N$, $G_t(\cdot)$ is quasi-continuous with respect to the scaled uniform topology.
\end{lem}
\begin{IEEEproof}
See Appendix~\ref{sec:app2}. In this induction-based proof we establish certain analytical properties of the workload map for all possible rate-regions that then allows us to obtain a quasi-continuous selection.
\end{IEEEproof}

Having established quasi-continuity, the next result demonstrates the almost compactness of the workload maps.
\begin{lem}\label{lem:almostcompact}
For $t\in \N$, both $G_t$ and $G^{\mathrm{wc}}_t$ are almost compact on $\R^{K}_+$ with respect to the scaled uniform topology.
\end{lem}
\begin{IEEEproof}
Since the workload at any time cannot exceed the amount work brought in from the last time the 
system was empty irrespective of the scheduler used, we automatically get the following bounds
\begin{align*}
G_t(\av|_{(0,t]})  \leq \av(0,t], \quad
G_t^{\mathrm{wc}}(\av|_{(0,t]})  \leq \av(0,t].
\end{align*}
This implies the
almost compactness of $G_t$ and $G_t^{\mathrm{wc}}$. Since
every sequence $\{\av^n\}$ converging to $\av$ (in $\D_{\muv}^K$ in the scaled uniform norm topology) converges point-wise too, it follows that $G_t(\av^n|_{(0,t]})$ and $G_t^{\mathrm{wc}}(\av^n|_{(0,t]})$ will lie in a bounded subset of $\R_+^K$ so we can use the Bolzano-Weierstrass theorem.
\end{IEEEproof}
Now, as already discussed, the proof of Theorems~\ref{thm:FiniteHorizonLDP} 
and~\ref{thm:FiniteHorizonLDP-NWC} is complete. 

Next, we discuss the LDP for the infinite-horizon workloads for the work-conserving max-weight scheduler with a simplex rate-region.

\subsection{LDP for Infinite-Horizon Workloads}
In this section, we establish an LDP of the sequence of the infinite-horizon workloads $\{\mcW^L = G^{\mathrm{wc}}(\Av^L)\}$  where $\Av^L \in \D_\muv^K$,
$\muv \in \text{int}(\RR_s)$, and $\{\Av^L\}$ satisfies Assumptions~\ref{Assumption1}-\ref{Assumption2}. 
Similar to the last section, we first show that 
the mapping $G^{\mathrm{wc}}
$ is quasi-continuous and almost compact on $\D^K_{\boldsymbol{\mu}}$ 
and then use Garcia's extended contraction principle to establish the desired LDP.

The following lemmas establish the necessary steps to apply Garcia's contraction 
principle to 
the stationary workload map. The first of these lemmas relates the 
infinite horizon workload mapping to that of a finite horizon. 

\begin{lem} \label{claimzero} 
Consider an arrival process $\av \in \D^K_\muv$. There exists a $s^*=s^*(\av) < \infty$  such that 
the workloads at time $-s^*$ under $\av$ falls within 
the rate region $\RR_s$, i.e.,
$G^{\mathrm{wc}}(\av|_{(s^*,\infty)}) \in \RR_s$. Furthermore, 
for any sequence of arrival processes $\{\av^n \in \D^K_\muv\}$ converging to
$\av$ (in scaled uniform topology), the workloads at time $-s^*$ under $\av^n$, when
$n$ is large enough, also fall within 
the rate region $\RR_s$, i.e., 
$\exists n_0$ such that $G^{\mathrm{wc},\RR_s}(\av^n|_{(s^*,\infty)}) \in \RR_s$ for $n> n_0$. 
\end{lem}

\begin{IEEEproof}
See Appendix~\ref{sec:app3}. The main idea is to use the fact that a suitably normalized sum (over all queues) workload process behaves like the workload of a single server queue with only one flow and a work-conserving service discipline. This allows us to use the continuity of the workload mapping
of a single server queue and the stability of the queue to arrive at the assertion of the lemma.
\end{IEEEproof}

Now we are ready to verify the requirements of Garcia's extended contraction principle. 

\begin{lem} \label{lem:quasi-continuous_G}
Let 
$\av \in \D^K_\muv$ be an arrival process with rate $\muv\in\mathrm{int}(\RR_s)$, and
  $G^{\mathrm{wc}}(\av)$ be the corresponding infinite-horizon 
workload. For any $\mcW  \in \sideset{^\av}{^{\mathrm{wc}}}\letG $ there exists a
sequence of arrivals $\{\av^n \in \D^K_\muv\}$ such that $\av^n$ converges
 to $\av$ in scaled uniform topology, $G^{\mathrm{wc}}(\av^n) \rightarrow \mcW$, 
 $G^{\mathrm{wc}}$ is continous at $\av^n$, and  $\Ish(\av^n) \rightarrow \Ish(\av)$. 
\end{lem}
\begin{IEEEproof}  
See Appendix~\ref{sec:app3}. The idea is to relate the infinite horizon workload, using Lemma~\ref{claimzero}, to a finite-horizon workload mapping whose quasi-continuity was established earlier.  
\ignore{{\color{magenta} We seem to be repeating too much material here. I think being brief is fine so we could drop the remainder that is in cyan. \color{cyan}
In particular, note that in lieu of~(\ref{eq:qdynamic}), Lemma~\ref{claimzero} implies that 
\(
G^{\mathrm{wc},\RR_s}(\av) = G_{s^*}^{\mathrm{wc}}(\av|_{(0,s^*]})
\) for  $n> n_0$. However,  the quasi-continuity of the finite-horizon workload mapping, 
$G_{s^*}$, and continuity of $\Ish_{s^*}$ allow for the  construction of finite horizon sequence 
$\av^n|_{(0,s^*]}$ so as to ensure the convergence of  $G_{s^*}(\av^n|_{(0,s^*]})$ and 
$\Ish_{s^*}(\av^n|_{(0,s^*]})$. Now, for all $n$, consider the sequence of arrivals $\av^n$ 
which consist of concatenation of  $\av|_{(s^*,\infty]}$ and $\av^n|_{(0,s^*]}$.
It is clear that this sequence converges to $\av$ in scaled uniform topology, hence 
\(
G^{\mathrm{wc},\RR_s}(\av^n) = G_{s^*}^{\mathrm{wc},\RR_s}(\av^n|_{(0,s^*]}) \rightarrow 
G^{\mathrm{wc},\RR_s}(\av) \). In addition, 
Assumptions~\ref{Assumption1}-\ref{Assumption2} ensure that
$\Ish(\av^n) \rightarrow \Ish(\av)$.
}}
\end{IEEEproof}
\begin{lem} \label{lem:almost-compact_G}
The mapping $G^{\mathrm{wc},\RR_s}$ is almost compact on $\D^K_{\vect{\mu}}$ with respect to the scaled uniform topology. 
\end{lem}
\begin{IEEEproof} 
See Appendix. The proof is along the same lines as that of Lemma~\ref{lem:quasi-continuous_G}.
\end{IEEEproof}
Again, the above lemmas and Garcia's extended contraction principle immediately yield the LDP for the sequence of the infinite-horizon workload in Theorem~\ref{thm:InfiniteHorizonLDP}. 

Let us now consider the problem of calculating the rate function. Eqn. \reqn{defn:J} suggests that the rate function $J$, where $ J(\bv) = \inf_{\av\in \D^K_\vect{\mu}: \sideset{^\av}{^{\mathrm{wc}}}\letG \ni \bv} \Ish(\av),$ could be interpreted as the minimum-cost solution among all paths $\av\in \D^K_\vect{\mu}$ such that $\bv \in \sideset{^\av}{^{\mathrm{wc}}}\letG$, where the cost of the path $\av$ is $\Ish(\av)$ and convex. Hence, the problem of finding the rate functions is a deterministic optimal control problem like those in \cite{Shakk08,Subra08}. 
However, 
the expressions for the rate functions $I_t$ and $J$ in \reqn{eq:defnIt} and \reqn{defn:J} are of little use in their current forms, as their computation is far from straight forward. In the next section, we simplify the rate functions when the arrival processes are limited to having \emph{i.i.d.} increments. 

\section{I.I.D. Increments: Simplified Rate Functions} \label{sect:RateFunction}

In this section, we give a calculation of the finite-horizon and infinite-horizon rate functions in the case when the arrivals have \emph{i.i.d.} increments. In this case, the cost of a sample path $\av\in \D^K$, which is $\Ish(\av)$, is additive and the total cost of any arrival sample path is the sum of the cost over all timeslots and queues. This property helps us to simplify the calculation of the rate functions. 

Consider the underlying arrival process $\Av$ to be a process with \emph{i.i.d.} increments, e.g., a compound Poisson arrival process with exponential packet length (see \cite{KittiHighSNR07}). 
For these \emph{i.i.d.} increment arrival processes, as mentioned earlier, we have
\begin{align*}
\Ish_t(\av|_{(0,t]})   =  \sum_{k=1}^K \sum_{i=1}^{t} \Lambda_1^*(x_i^k), \text{ and }
\Ish(\av)  = \sum_{k=1}^K \sum_{i=1}^{+\infty} \Lambda_1^*(x_i^k).
\end{align*}
Next, we simplify these rate functions.

\subsection{Infinite-Horizon Rate Function}
The following lemma expresses the infinite-horizon rate function $J$ as the infimum of the finite-horizon rate functions $I_t$ over all time $t$. 

\begin{lem} \label{lem:simplified_J}
For \emph{i.i.d.} increment arrival processes with $\boldsymbol{\mu} \in \mbox{int}(\RR_s) $, the infinite-horizon rate function $J$ is simplified as \beq J(\bv) = \inf_{t\ge 1} I_t(\bv).\eeq
\end{lem}
\begin{IEEEproof}
The cost of a sample path over time is the sum of the cost of arrivals in all timeslots. As in the proof of Lemma~\ref{lem:quasi-continuous_G}, for $\av \in \D^K_\muv$ where $\boldsymbol{\mu} \in \mbox{int}(\RR_s)$, we can find $t := s^*(\av)$ such that $\Wv_{t}(\av) \in \RR_s$. Hence, for $\av$ such that $ \bv\in\sideset{^\av}{^{\mathrm{wc}}}\letG$, one can reduce the cost of the path by constructing a new sample-path $\tilde{\av}$ by setting $\tilde{\av}_v = \muv$ for all $v > t$ and $\tilde{\av}_v=\av_v$ for all $v\leq t$ while still satisfying $\bv \in\sideset{^{\tilde{\av}}}{^{\mathrm{wc}}}\letG $. This is because $\Ish(\tilde{\av}) = \Ish_{t}(\av|_{(0,t]})\leq \Ish(\av)$. On the other hand, since $\Wv_{t}(\av) \in \RR_s$, we can write $\bv \in \sideset{^{\av|_{(0,t]}}}{^{\mathrm{wc}}_t}\letG$. All of these imply that   
\begin{align*} 
J(\bv) = \inf_{\av\in \D^K_\muv: \sideset{^\av}{^{\mathrm{wc}}}\letG \ni \bv} \Ish(\av)  = \inf_{t\ge 1} \ \inf_{\xv \in \R^{Kt}_+: \sideset{^\xv}{^{\mathrm{wc}}_t}\letG \ni \bv} \Ish_t(\xv) = \inf_{t\ge 1} I_t(\bv),
\end{align*} 
 by the definition of $I_t(\bv)$ in \reqn{eq:defnIt}.
\end{IEEEproof}
With this simplification available, we now look at the finite-horizon rate function $I_t$ in more details.

\subsection{Finite-Horizon Rate Function}

In this subsection, we provide a further simplified expression of the finite-horizon rate function $I_t$. 

\begin{lem} \label{lem:RateFunction}
For $t\in \N$, the finite-horizon rate function $I_t$ is simplified as  
\beq \label{eq:It1} 
I_t(\bv) = \min\left(\Ish_1(\bv),\min_{u\in (1,t]} \ \inf_{\xv \in \A(u,\bv)} \Ish_u(\xv) \right)
\eeq
for $\bv \in \R^K_+$, where for $u>1$
\begin{align} 
\A(u,\bv) := \{\xv \in \R^{K\times u}_+: \bv \in \sideset{^\xv}{^{\mathrm{wc}}_u}\letG, G_{u-v}^{\mathrm{wc}}(\xv|_{(v,u]})\not\in \RR_s, \forall v\in [1,u-1]\}.\label{eq:Aub}
\end{align}
\end{lem}
\begin{IEEEproof} This follows the idea from the proof of Lemma~\ref{lem:simplified_J}. Let $t \in \N$.
For $\av\in \D^K_\muv$ such that $\bv \in \sideset{^\av}{^{\mathrm{wc}}_t}\letG$, we let 
\begin{align*}
u =  \min\bigg\{t, \min\Big\{s\in [1,t-1]:  \Wv_s = G_{t-s}^{\mathrm{wc}}(\av|_{(s,t]}) \in \RR_s\Big\}\bigg\}.
\end{align*}
In other words, $-u$ is the last time the workload vector is inside the capacity region $\RR_s$ before time $0$.  With this definition of $u$, for all $u>1$ we have $\Wv_{v} \not\in \RR_s$ for all $v\in [1,u-1]$. By definition of $I_t$, we already know that the workload vector starts initially inside $\RR_s$ at time $-t$. Therefore, we can find another path $\tilde{\av} \in \R^{Kt}_+$ with a reduced cost while keeping the workloads at time $-u+1$ to $0$ (i.e., $\Wv_{u-1}$ to $\Wv_0$) intact by setting $\tilde{\av}_v = \muv, \forall v\in (u,t]$ and $\tilde{\av}_v = \av_v$ otherwise. It is easy to see that we have $\Ish_t(\av|_{(0,t]}) \ge \Ish_u(\av|_{(0,u]}) = \Ish_{t}(\tilde{\av}|_{(0,t]})$ and yet $\bv \in \sideset{^{\tilde{\av}|_{(0,u]}}}{^{\mathrm{wc}}_u}\letG = \sideset{^{\tilde{\av}}}{^{\mathrm{wc}}_t}\letG$. The same logic also applies to the case when $u=1$. However, in this case the only way that we can achieve a workload vector $\bv$ at $0$ is if exactly that amount of work arrives, i.e., if $\av_1=\bv$.

Since by definition $\Wv_v = G_{u-v}^{\mathrm{wc}}(\av|_{(v,u]})$ for $v\in [1,u-1]$ (when $u>1$), we have 
\begin{align*} 
I_t(\bv) = \inf_{\xv \in \R^{Kt}_+: \sideset{^\xv}{^{\mathrm{wc}}_t}\letG \ni \bv}  \Ish_t(\xv|_{(0,t]}) & = \min\Biggl(\Ish_1(\bv),
\min_{u\in (1,t]} \ \inf_{\substack{\xv \in \R^{Ku}_+: \bv \in  \sideset{^\xv}{^{\mathrm{wc}}_u}\letG , \\G^{\mathrm{wc}}_{u-v}(\xv|_{(v,u]}) \not\in \RR_s}} \Ish_u(\xv)\Biggr)\\
& = \min\left(\Ish_1(\bv),\min_{u\in (1,t]} \ \inf_{\xv \in \A(u,\bv)} \Ish_u(\xv|_{(0,u]})\right),
\end{align*}
where $\A(u,\bv)$ is defined as in \reqn{eq:Aub}.
\end{IEEEproof}

\begin{note} The above lemma reduces the set of feasible sample paths to the set $\A(u,\bv)$ for $u\in (0,t]$. 
It is interesting to note the property of the sample-paths in this set. For any $\xv \in \A(u,\bv)$, we have $\hat{\Wv}_0(\xv) = \hat{\xv}(0,u] - (u-1) = \hat{\bv}$, recalling that the $\hat{\cdot}$ notation is the normalized sum of the elements of the vectors. There is no wastage of service capacity over the $u-1$ timeslots because $\forall v\in [1,u-1]$, $\Wv_v = G^{\mathrm{wc}}_{u-v}(\xv|_{(v,u]})\not\in \RR_s$ and hence $\hat{\Wv}_v > 1$.  That is, any sample path $\xv \in \A(u,\bv)$ has its normalized sum of the arrivals over time $(0,u]$ and queues equal to $\hat{\xv}(0,u] =  \hat{\bv} + (u-1)$. 
\end{note}

In addition, an immediate implication of Lemma \ref{lem:RateFunction} is that we can rewrite $J$ in  \reqn{defn:J} as 
\begin{align} \label{eq:revisedJ} 
J(\bv) = \inf_{t\ge 1} \ I_t(\bv)=\inf\Big(I_1(\bv),\inf_{t\ge 2} I_t(\bv)\Big) & = \inf\left(\Ish_1(\bv),\inf_{t\ge 2}  \min_{u\in (1,t]} \ \inf_{\xv \in \A(u,\bv)} \Ish_u(\xv|_{(0,u]})\right)\nn\\ 
& = \inf\left(\Ish_1(\bv),\inf_{t\ge 2}  \inf_{\xv \in \A(t,\bv)} \Ish_t(\xv|_{(0,t]}\right),
\end{align} 
where we also used the fact that $I_1(\bv)=\Ish_1(\bv)$.
If we denote $t^*$ as the optimizer of the last equation, then $t^*$ is called the \emph{critical timescale} (see \cite{Wisch01}). 
It can then be interpreted that $t^*$ is the most likely length of time it would take to ``fill" the buffers to a given level $\bv$ from being ``empty'' (more precisely, anywhere within $\RR_s$). 

\begin{note}\label{rem:rate_calc}
The induction-based proof of Lemma~\ref{lem:continuity_Gt} reveals another important property of the sample-paths, namely, that at every stage it suffices to consider the quasi-continuous selection $H^{\mathrm{wc}}(\cdot)$ of the scheduling function. This then implies that we need to consider all valid arrivals sequences that respect the constraints of the set $\A(\cdot,\bv)$ such that using any of the allowed (based on the workload vector at each time) scheduling actions given by $H^{\mathrm{wc}}(\cdot)$ results in the workload vector $\bv$ at the required epoch.
\end{note}

Note that for fixed $u\in \N$, $\inf_{\xv \in \A(u,\bv)} \Ish_u(\xv|_{(0,u]})$ is a optimization problem, with a convex cost function $\Ish_u(\cdot)$ and a set $\A(u,\bv)$ of allowable solutions, in general, a $K(u-1)$-dimensional set. This problem is difficult to solve analytically. Since the cost function $\Ish_u(\xv|_{(0,u]}) = \sum_{i=1}^u \sum_{k=1}^K\Lambda_1^{*,k}(\xv_i)$ is additive, a possible numerical method is the numerical backwards induction of dynamic programming. However, the method suffers from the curse of dimensionality and hence is not practical for large $u$ and $\bv$.  
Hence, we turn our attention to finding some simplified bounds of the rate functions. This can be done by employing the additivity and convexity of the rate function $\Ish_t$. Next we present some bounds for the case when $K=2$. 

\subsection{Properties of the Minimum-Cost Sample Paths}
Here, we see that the convexity of the cost function $\Lambda_1^{*,k}$ for all $k\in\K$, induces two properties for the optimal paths. 

\begin{property} \label{Property1}
\emph{Constant-speed linear path is the cheapest}. Among all arrival sample paths $\xv \in \D^K_\muv$  with the only constraint that $\xv(0,t] = \bv$, i.e., the total amount of arrivals at the end of time $t\in\N$ equals $\bv$, the cheapest or minimum-cost path is the constant-speed linear path, where the arrival in each timeslot is equal to $\bv/t$. 
\end{property}
\begin{IEEEproof}
This is because the path cost function is additive, i.e., $\Lambda^*_t(\xv) = \sum_{k=1}^K\sum_{i=1}^t  \Lambda_1^{*,k}(x_i^k)$, and the per-timeslot cost function $\Lambda_1^{*,k}$ is convex.  
Applying Jensen's inequality \cite{Rocka70} gives 
\begin{align*}
\Lambda^*_t(\xv) & = \sum_{k=1}^K\sum_{i=1}^t  \Lambda_1^{*,k}(x_i^k) \ge \sum_{k=1}^K t \Lambda_1^{*,k}\left(\frac{1}{t} \sum_{i=1}^t x_i^k\right) =  \sum_{k=1}^K t \Lambda_1^{*,k}(b^k/t),
\end{align*}
with equality when $x_i^k = b^k/t$ for all $i$ and $k$. See an illustration in Figure~\ref{fig:Prop1}. 
\end{IEEEproof}

From now onwards, without loss of generality, we assume that the arrivals, workloads and service vectors are scaled by $1/\mathbf{C}$ where $\mathbf{C}=(C^1,\dotsc,C^K)$ is the vector of maximum service rates. In this normalized setting, the maximum service rate of all the users is $1$. Assume also that $\Lambda_1^{*,k}$ for $k\in\K$ is suitably modified for this scaling, and, for ease of exposition, that with this scaling the processes are statistically identical, i.e., $\Lambda_1^{*,k}\equiv\Lambda_1^{*}$ for all $k\in\K$. We can then write down the following property. 
\begin{property} \label{Property2}
\emph{Constant-speed linear path closest to the equal line is the cheapest.} 
For constant-speed linear paths $\av \in \R^{Kt}_+$ with the sum $\hat{\av}(0,t] = d$, 
the cost of the path is a Schur-convex function~\cite{MarshallOlkin1979}.
\end{property}
\begin{IEEEproof}
Since the arrival paths are constant-speed linear path, without loss of generality we can consider arrival paths in a single timeslot. Consider path $\xv\in\R_+^K$, then the cost of the path is $\sum_{k=1}^K \Lambda_1^*(x^k)$ where $\Lambda_1^*(\cdot)$ is a convex function. Therefore from the results in \cite{MarshallOlkin1979} it follows that $\sum_{k=1}^K \Lambda_1^*(x^k)$ is a Schur-convex function. In order words, if $\xv$ is majorized by $\yv\in\R_+^K$ (denoted by $\xv \prec\yv$), i.e., $\sum_{k=1}^j x^{[i]} \leq \sum_{k=1}^j y^{[i]}$ for $j=1,\dotsc,K-1$ and $\sum_{k=1}^K x^k=\sum_{k=1}^K y^k$ where $x^{[i]}$ is the $i^{\mathrm{th}}$ largest component of $\xv$, then $\sum_{k=1}^K \Lambda_1^*(x^k)\leq\sum_{k=1}^K \Lambda_1^*(y^k)$. 

This is easily appreciated when $K=2$. Let $\xv=(x,d-x) \in \R^2_+$ and $\yv=(y,d-y)\in \R^2_+$, where $y>x>d/2$, then we have $\Lambda_1^{*}(x)+\Lambda_1^{*}(d-x) \le \Lambda_1^{*}(y)+\Lambda_1^{*}(d-y)$, and $\xv$ is cheaper than $\yv$. We illustrate this in Figure~\ref{fig:Prop2}.
\end{IEEEproof}

\begin{figure}[tbp]
\centering
    \subfigure[Property 1]{
        \includegraphics[width=.75\linewidth]{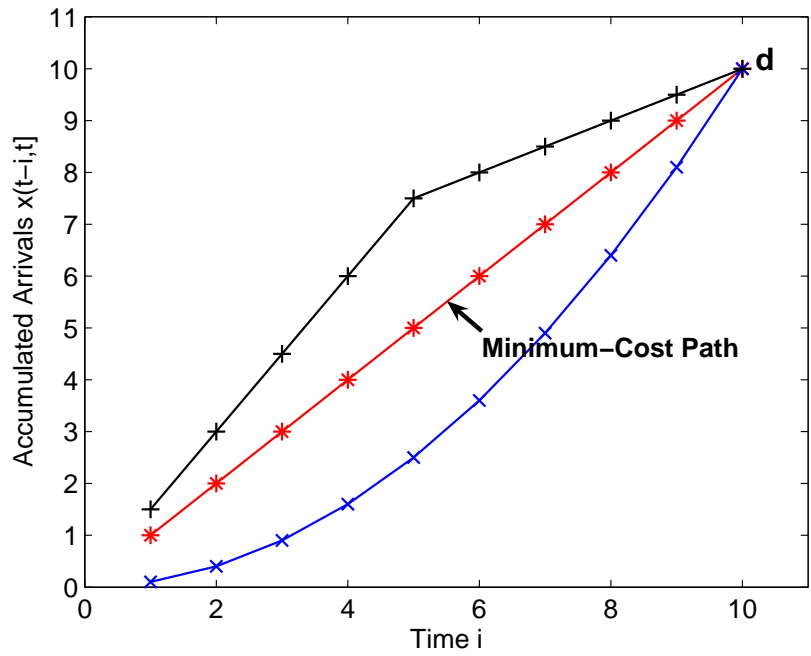} \label{fig:Prop1}}
        \hfill
    \subfigure[Property 2]{
        \includegraphics[width=.75\linewidth]{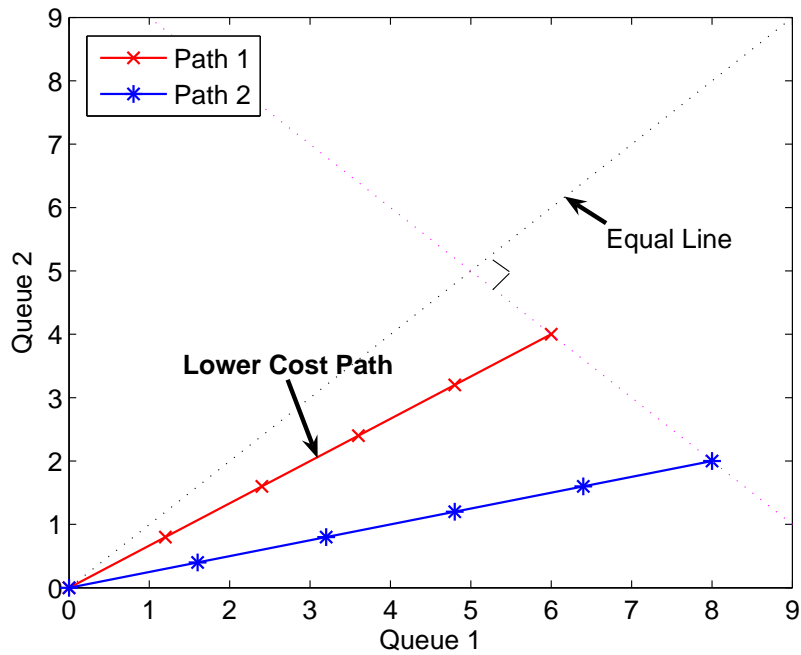} \label{fig:Prop2}}
    \hfill
\caption[Two properties of the minimum-cost sample paths]{Two properties of the minimum-cost sample paths: (a) Property 1: the minimum-cost path is the constant-speed linear path. 
(b) Property 2: 
 Path 1 which is closer to the Equal Line has a lower cost than Path 2.} 
\end{figure}

These properties are also used in \cite{Berts98,Ying06,Shakk08} for large-deviations analysis of scheduling disciplines. Next, we use these properties to calculate $I_2$ and bounds on $I_t$ for $t \in \N$ for just the work-conserving scheduler operating on the simplex rate-region. 

\subsection{Example: Calculation of $I_2$}
Here we look at an example for calculation of the finite-horizon rate function $I_t$ to illustrate that the calculation continues to be rather involved. 
For simplicity, consider the case when $t=2$ and $K=2$. 
From \reqn{eq:It1}, $I_2(\bv)$ for $\bv\in \R^2_+$ can be written as 
\begin{align} 
 I_2(\bv)  =  &  \min\Big\{\sum_{i=1}^2\Lambda_1^{*}(b^i), \inf_{\xv|_{(0,2]} \in \A(2,\bv)} \sum_{i=1}^2\sum_{k=1}^2\Lambda_1^{*}(x_i^k)\Big\},\label{eq:I2a} 
\end{align}
where 
$$ 
\A(2,\bv) = \left\{\av|_{(0,2]} \in \R^4_+: \av_2 \not\in \RR_s, \bv \in \sideset{^{\av|_{(0,2]}}}{^{\mathrm{wc}}_2}\letG\right\}.
$$ 
The workload at time zero is $\Wv_0 = G_2^{\mathrm{wc}}(\av|_{(0,2]}) = \av_1 + [\av_2-H^{\mathrm{wc}}(\av_2)]^+$, which is equal to $\av(0,2] -H^{\mathrm{wc}}(\av_2)$ since $\av_2 \not\in \RR_s$. On the other hand, we require $\Wv_0 =\bv$. Hence, using the scheduler $H^{\mathrm{wc}}$ given in \reqn{eq:Hwc}, 
we can express $\A(2,\bv)$ as $\A(2,\bv) = \A_{(1)} \cup  \A_{(2)} \cup  \A_{(3)},$ where $\A_{(j)} \subseteq \R^4_+, j=1,2,3$, are defined as 
\begin{align*}
\A_{(1)} &:= \{(\av_1,\av_2) \in \R^4_+: a_2^1\ge a_2^2, a_2^1\ge 1,  \av(0,2]  = \bv + (1,0)\},\\
\A_{(2)} &:= \{(\av_1,\av_2)\in \R^4_+: a_2^1 \le a_2^2, a_2^2\ge 1,  \av(0,2]  = \bv + (0,1)\},\\
\A_{(3)} &:= \{(\av_1,\av_2)\in \R^4_+: a_2^1 \le 1, a_2^2 \le 1, a_2^1 + a_2^2 \ge 1,  \av(0,2] = \bv + \text{Proj}_{\RR_s}(\av_2)\}.
\end{align*} 
Note that in the definition of $\A_{(1)}$ and $\A_{(2)}$ we have used the property highlighted in Remark~\ref{rem:rate_calc}. For the two user case the scheduling function $H^{\mathrm{wc}}(\cdot)$ is only discontinuous at $\xv\in\R^2_+$ such that $x^1=x^2\geq 1$. Here one can either choose the service vector $(1,0)$ or $(0,1)$ to obtain a quasi-continuous selection. Thus, both of these options have to be considered.

Using these newly defined sets the second term in the RHS of \reqn{eq:I2a} can be rewritten as 
\begin{align*}
\inf_{\xv \in \A(2,\bv)} \sum_{i=1}^2\sum_{k=1}^2\Lambda_1^{*}(x_i^k)  =\min_{j\in [1,3]} \inf_{\xv \in \A_{(j)}}\sum_{i=1}^2\sum_{k=1}^2\Lambda_1^{*}(x_i^k).
\end{align*}
Trajectories of some examples of the (accumulated) arrival sample paths are illustrated in Figure~\ref{fig:I2a} and their corresponding workload trajectories in Figure~\ref{fig:I2b}. In particular, Figure~\ref{fig:I2a} shows example trajectories of the accumulated arrival sample paths $\av|_{(0,2]}\in \A_{(j)}, j=1,2,3$, in the calculation of $I_2(\bv)$, where $\bv = (4,2)$. Also, for particular values of $\av|_{(0,2]} \in \A_{(1)}$, namely, $\av_{2} = (2.5,1)$ and $\av(0,2]= (5,2)$ Figure~\ref{fig:I2b} shows the workload paths $\Wv^{(1)}$. The same figure also displays examples corresponding to arrival paths in $\A_{(j)}, j=2,3$. For the specific example from $\A_{(1)}$ the figure shows that $\Wv^{(1)}_1 = \av_{2} = (2.5,1)$ and $\Wv^{(1)}_0 = \av(0,2] - (1,0) = (4,2)$.

This example underlines the difficulty in finding the rate function even for small timescales. We expect that the number of constrained sets like $\A_{(j)}$ will grow exponentially with time duration $t$. However, the example gives us some insight on how to find some simple upper and lower bounds to $I_t$ for any $t\in \N$. 

\begin{figure}[tbp]
\centering
    \subfigure[Trajectories of Arrival Paths]{
        \includegraphics[width=.75\linewidth]{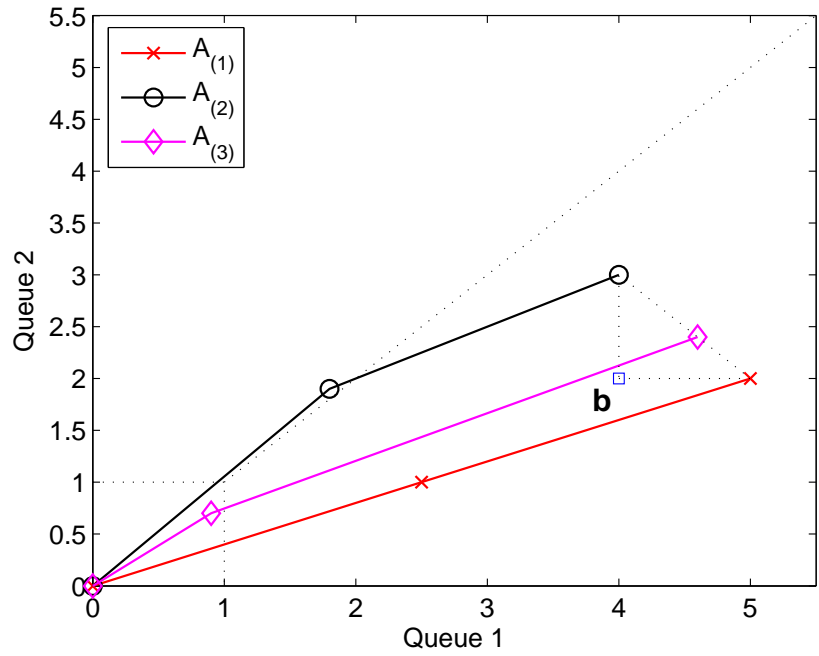} \label{fig:I2a}}
        \hfill
    \subfigure[Trajectories of Workload Paths]{
        \includegraphics[width=.75\linewidth]{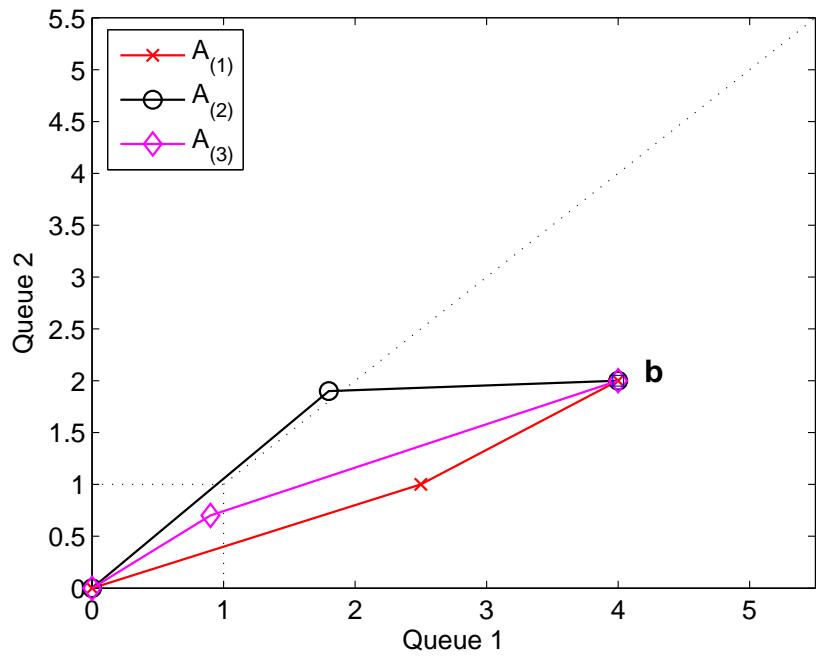} \label{fig:I2b}}
    \hfill
\caption[Example of accumulated arrival and workload paths for $I_2$]{Example of accumulated arrival and workload paths for calculation of $I_2(\bv)$.} 
\label{fig:I2} 
\end{figure}

\subsection{Bounds on $I_t$}
In this subsection, we find simple expressions that give lower or upper bounds to $\inf_{\xv \in \A(u,\bv)} \Ish_u(\xv)$, which in turn give the bounds to $I_t$ and $J$. We focus on $K=2$ but similar results can be obtained for general $K$.

\begin{lem} \label{lem:MSLDPItBounds}
For $K=2$, $\bv \in \R^2_+$, $I_t(\bv)$ can be bounded as 
\beq I_t(\bv) \ge \min_{u\in (0,t]} u\sum_{k=1}^K\Lambda_1^*\left(\frac{1}{u} \Big(\text{Proj}_{\XXX(u,\bv)}(\vect{0})\Big)^k\right) \label{eq:ItLB}\eeq
and when $\bv \not\in [0,1)^2$, \beq I_t(\bv) \le \min_{u\in (0,t]} u\sum_{k=1}^K\Lambda_1^*\left(\frac{1}{u}(b^k+(u-1)H(\bv)^k)\right), \label{eq:ItUB}\eeq
where the convex set $\XXX(u,\bv) \subseteq \R^2_+$ is defined as
\begin{align} 
\XXX(u,\bv) :=  \{\bv+\mathbf{v}:  \quad v^1+v^2=(u-1), v^1,v^2\ge 0\}.
\end{align}
\end{lem}
\begin{IEEEproof} See Appendix~\ref{sec:app4}. \end{IEEEproof}

Next we look at the tightness of the above bounds for an example of compound Poisson source process with exponential packet size. 
We expect the tightness to depend on the traffic load. 

\subsection{Numerical Examples} 

Here we compare the finite-horizon rate functions $I_2$ for three schedulers, namely, max-weight, GPS with equal weights, and a priority scheduler that gives higher priority to queue $1$. We also examine the tightness of the bounds given in Lemma~\ref{lem:MSLDPItBounds}. Both of these are for an average of \emph{i.i.d.} compound Poisson source processes with exponential packet size where the packet arrivals follow Poisson distribution of rate $\lambda$ and the average packet size is $1/\mu$ (see \cite{KittiHighSNR07}). 
The function $\Lambda^*$ for this process is given by $$ \Lambda^*(x) = \mu (\sqrt{x}-\sqrt{\lambda})^2,$$ for $x\in \R_+$. Note that this has domain $\R_+$ and therefore satisfies Proposition~\ref{prop1}. We once again make the simplifying assumption that the processes for the different users are statistically identical and that the rate-region is the unit simplex.

First we present results for the comparison of the different scheduling policies. Here we set $\lambda = 0.1$ or $0.3$ with the average packet size of $1/\mu = 100$. Fig.~\ref{fig:I2mwl1} and \ref{fig:I2mwl3} show the finite-horizon (two-timestep) rate function $I_2(\bv)$ for $\lambda = 0.1$ and $0.3$ respectively. However, these calculations are best appreciated when we compare them to the rate functions of other well-known scheduling policies. Fig.~\ref{fig:I2mwgps} compares the max-weight scheduler with a GPS scheduler with equal weights both at $\lambda=0.3$. One can see that the rate function for the max-W scheduler is greater than the rate-function of the GPS scheduler for some range of $\bv$, i.e., where $b^1$ is much greater that $ b^2$ and {\it vice versa}. This means that for this range of $\bv$ and in the many-sources asymptotic sense, the work-conserving MW scheduler performs better than the work-conserving GPS scheduler (when we only consider the two time-step workload). The reason is that with $b^1$ much greater than $b^2$, MW can serve queue~$1$ at full capacity $(1)$, where GPS has to serve both queues equally at $1/2$. Hence, the workload under MW is less likely to reach $(b^1,b^2)$ at the end of the two time-slots when the system starts from being empty. Similarly \ref{fig:I2mwp} compares the rate-functions for the max-weight scheduler and a priority scheduler where user $1$ has higher priority. Since the max-weight policy does not discriminate between the two users, we find that it is less likely for $b^2$ to large for the max-weight in comparison to the priority policy with the reverse being true for $b^1$.

\begin{figure}[tbp]
\centering
    \subfigure[$\lambda=0.1$]{
        \includegraphics[width=.7\linewidth]{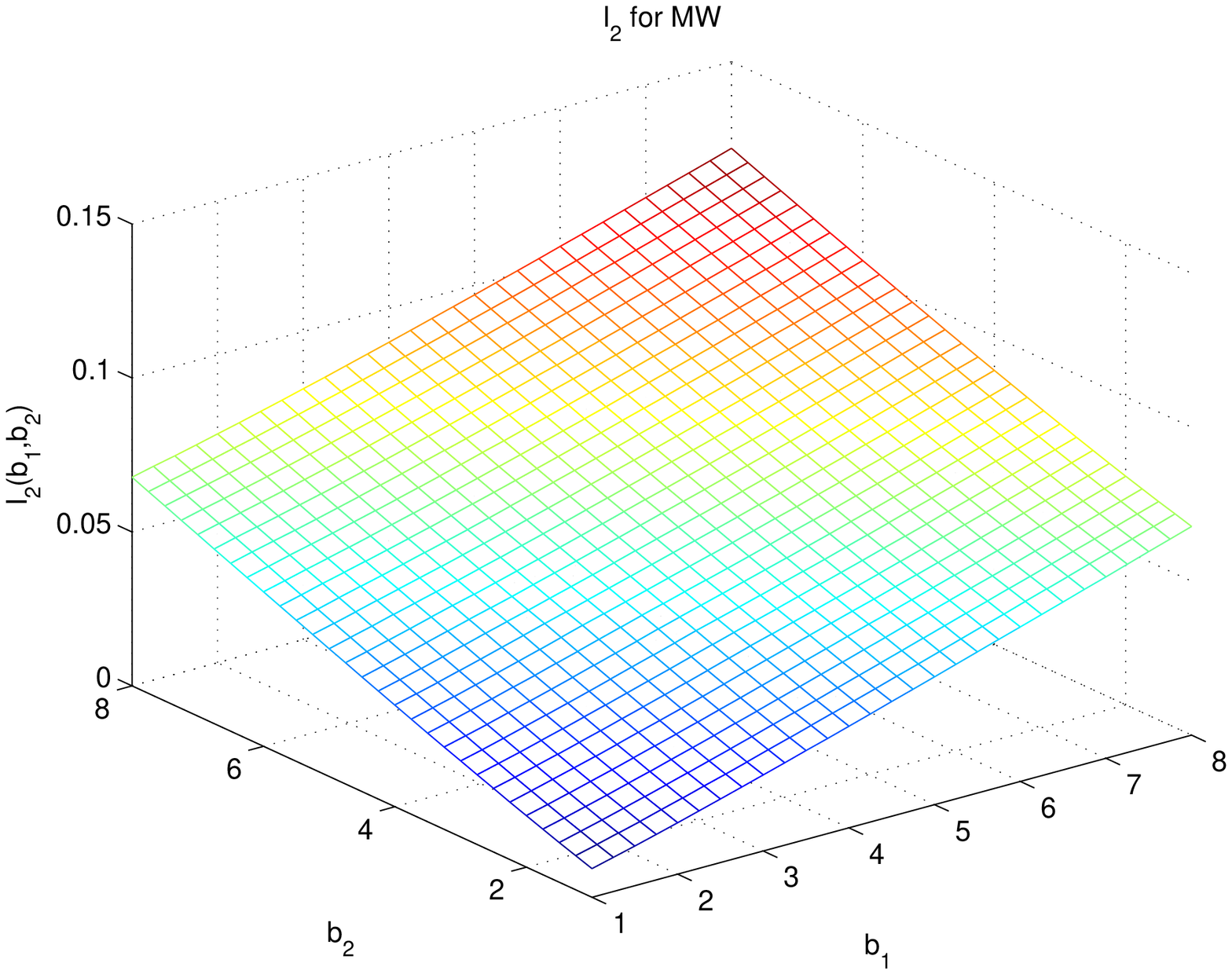} \label{fig:I2mwl1}}
        \hfill
    \subfigure[$\lambda=0.3$]{
        \includegraphics[width=.7\linewidth]{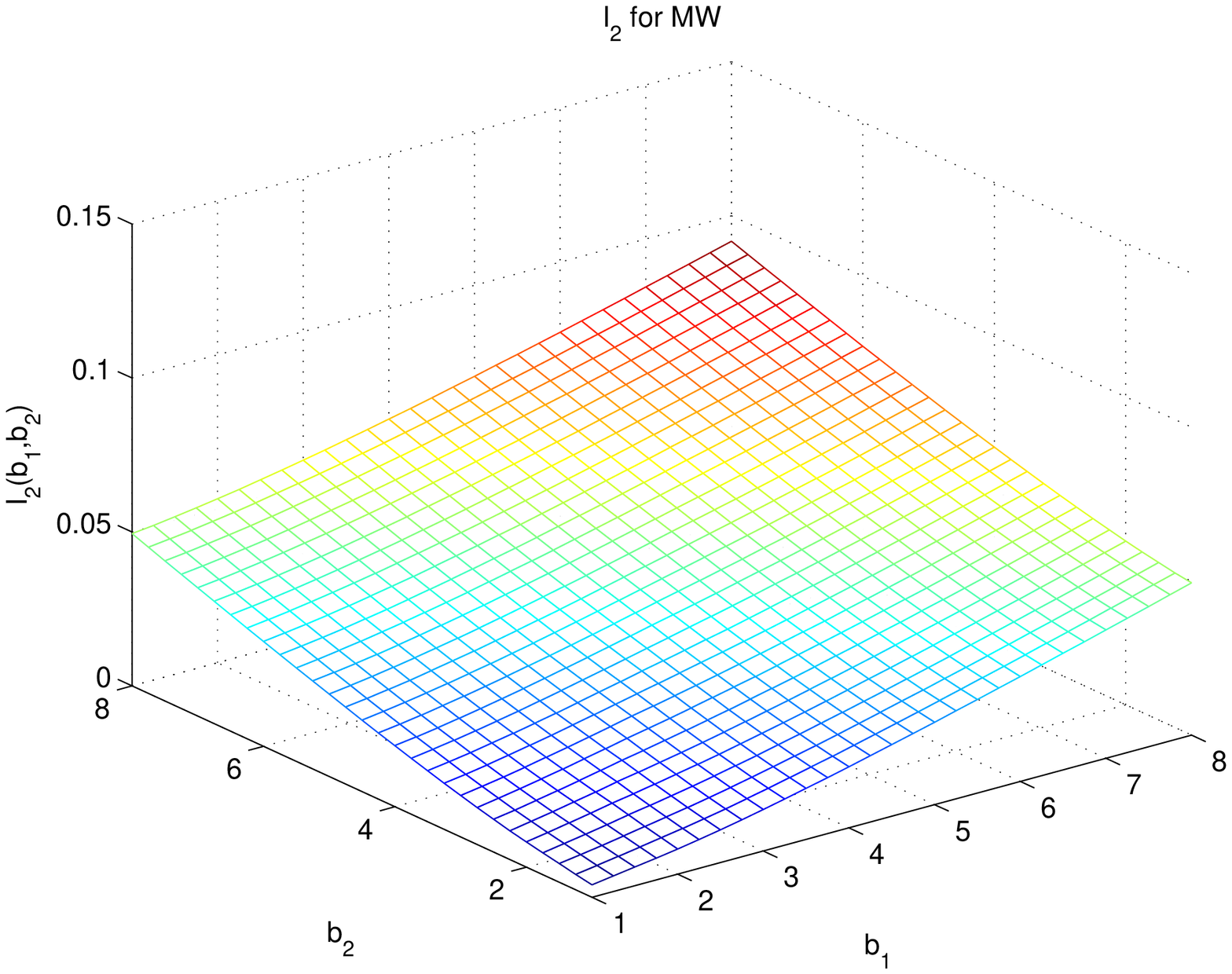} \label{fig:I2mwl3}}
    \hfill
\caption{Finite-horizon rate function $I_2(\bv)$ for the max-weight scheduler.} 
\label{fig:I2mw} \vspace{-0.1in}
\end{figure}

\begin{figure}[tbp]
\centering
    \subfigure[Max-weight versus GPS]{
        \includegraphics[width=.7\linewidth]{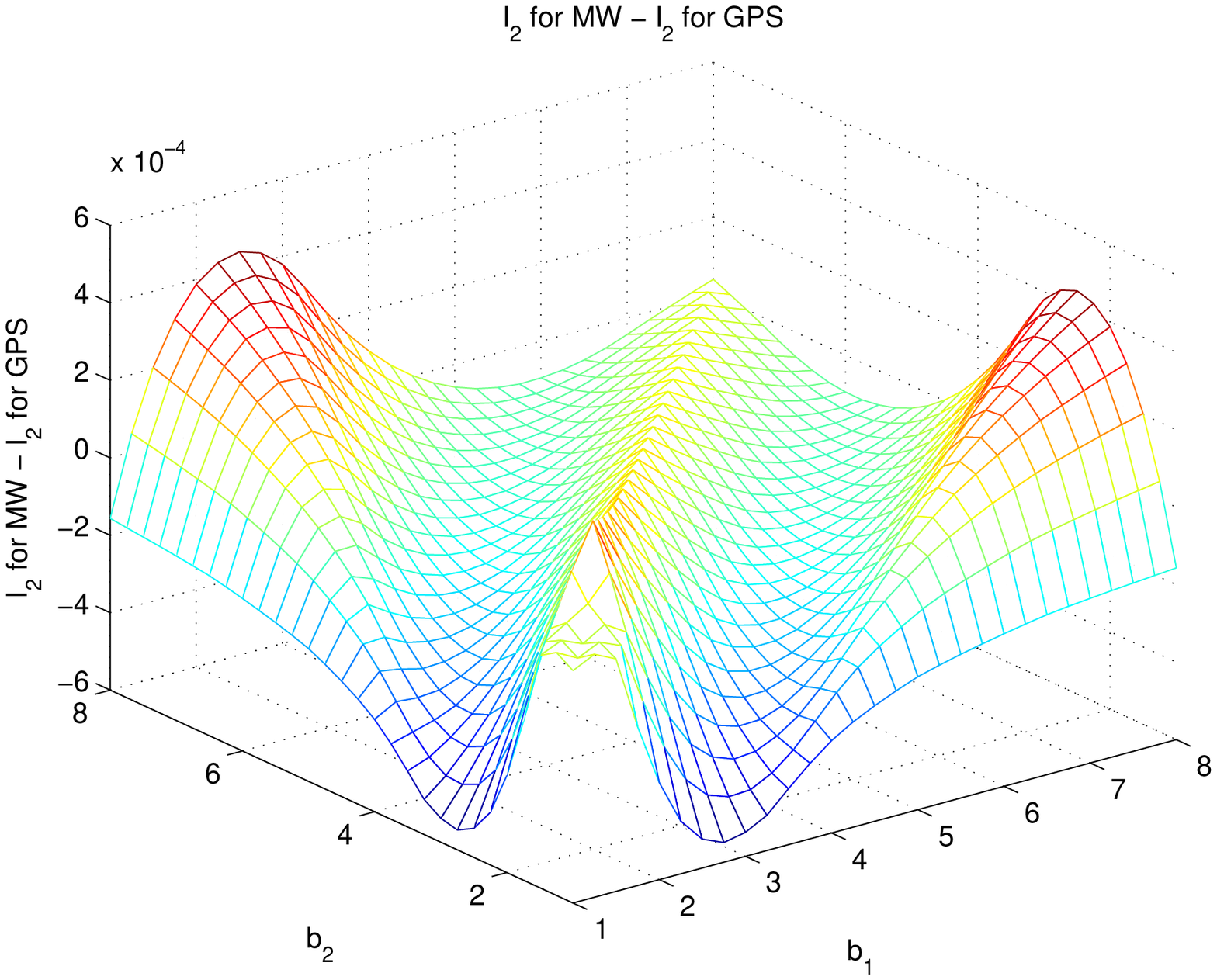} \label{fig:I2mwgps}}
        \hfill
    \subfigure[Max-weight versus Priority for user $1$]{
        \includegraphics[width=.7\linewidth]{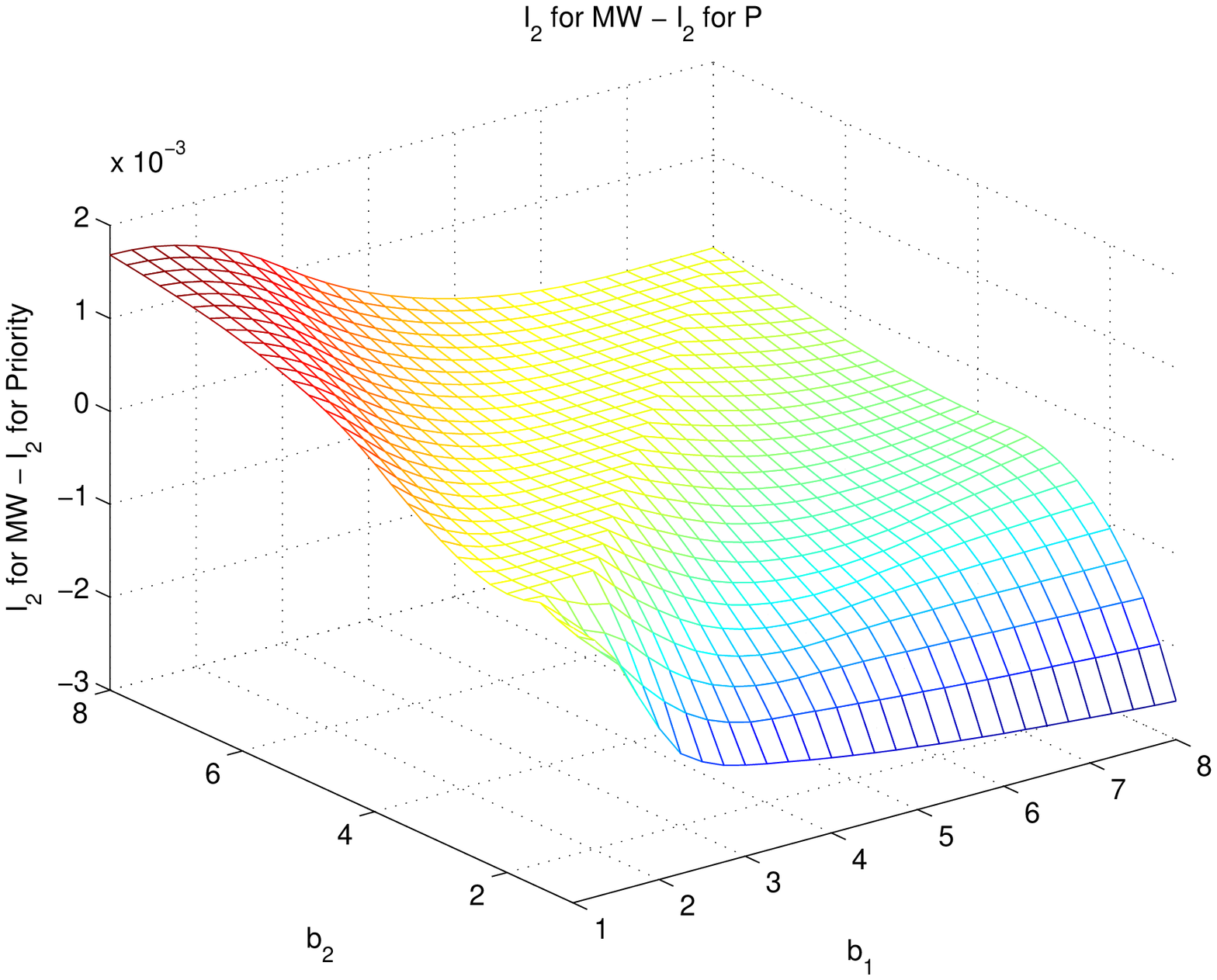} \label{fig:I2mwp}}
    \hfill
\caption{Comparison of the rate-functions of the max-weight scheduler, the GPS scheduler with equal weights and the priority scheduler where user $1$ has higher priority.} 
\label{fig:I2compare} \vspace{-0.1in}
\end{figure}

Next we compare our bounds for the rate function of the max-weight policy with the exact rate function in the scenarios where we can calculate it by by brute force. Figure \ref{fig:Bound} shows the upper and lower bounds and the actual values of $I_t,$ for $t=10$, at $\mu = 0.01$, and various values of $\lambda = 0.1,0.2,0.3$ and when $\bv = (b^1,b^2=1)$ for various values of $b^1$. Figure \ref{fig:BoundT} shows the corresponding minimizing $t^*$ for the bounds and the actual expression of $I_t$. We note that for all $\bv$ in this example, $J(\bv)$ is actually equal to $I_t(\bv)$ for $t=10$ since all optimizing $t^*$ is less than 10 (see \reqn{eq:revisedJ}). This example shows that, in the range of $\bv$ in consideration, both bounds are tight and almost coincide when the traffic load is small, i.e., $\lambda=0.1$. 
However, when the traffic load is higher, the lowerbound becomes loose while the upperbound is still considerably tight. 

It is interesting to note the optimal timescale $t^*$ which the queues most likely to take to reach the level $\bv$. Figure \ref{fig:BoundT} shows that, for example, it is most likely to take only two timeslots for CPE process with $\lambda = 0.2$ to reach the buffer level $\bv = (3,1)$, while the most likely timescale is four timeslots when the traffic load is higher ($\lambda = 0.3$). Figure \ref{fig:BoundTrajec2} and Figure \ref{fig:BoundTrajec} show the optimal trajectories of the accumulated arrival process and the workload process for $\lambda = 0.2$ and $0.3$, respectively.



\begin{figure}[tbp]
\centering
    \subfigure[Rate functions]{
        \includegraphics[width=.75\linewidth]{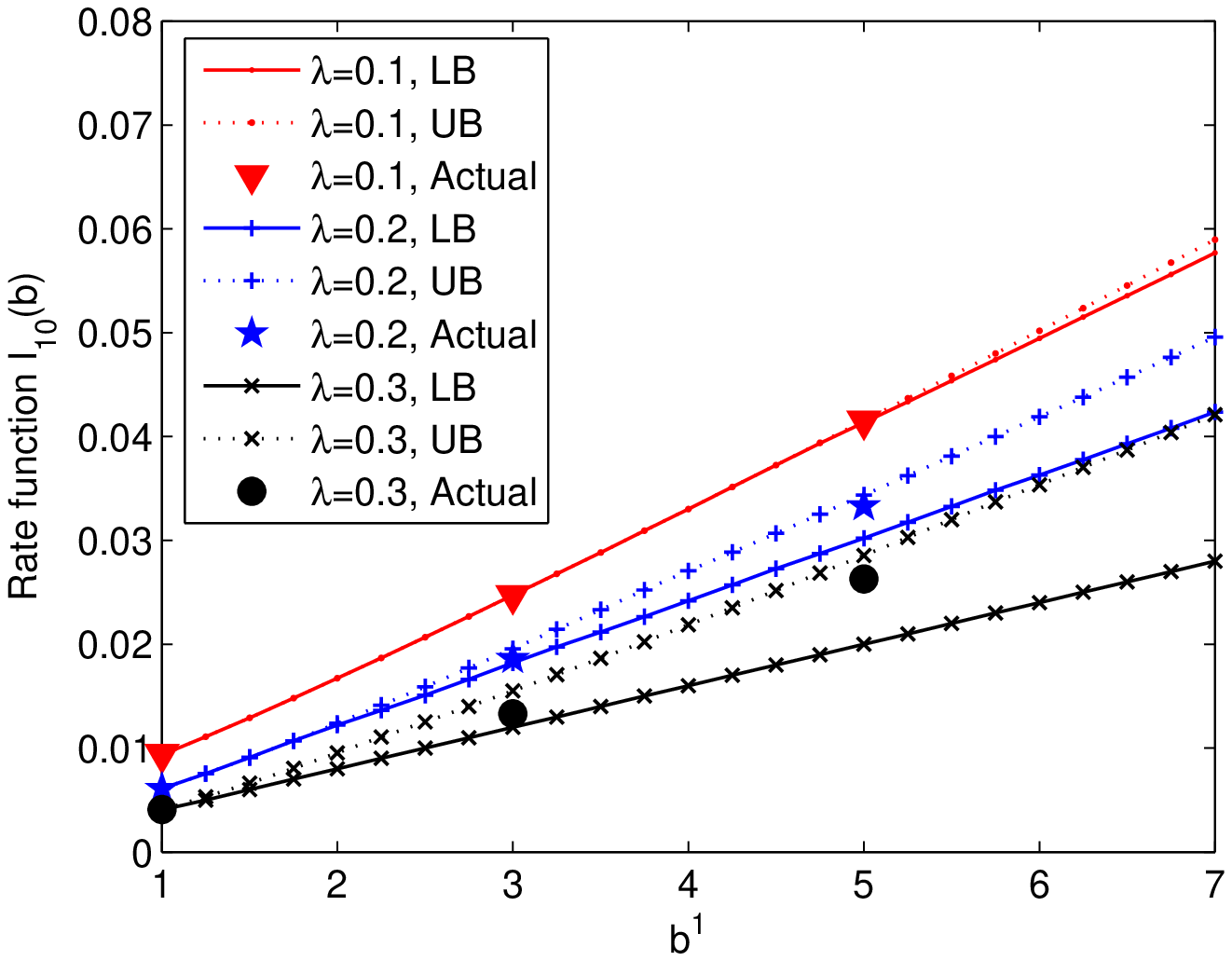} \label{fig:Bound}}
        \hfill
    \subfigure[Optimal Timescales]{
        \includegraphics[width=.75\linewidth]{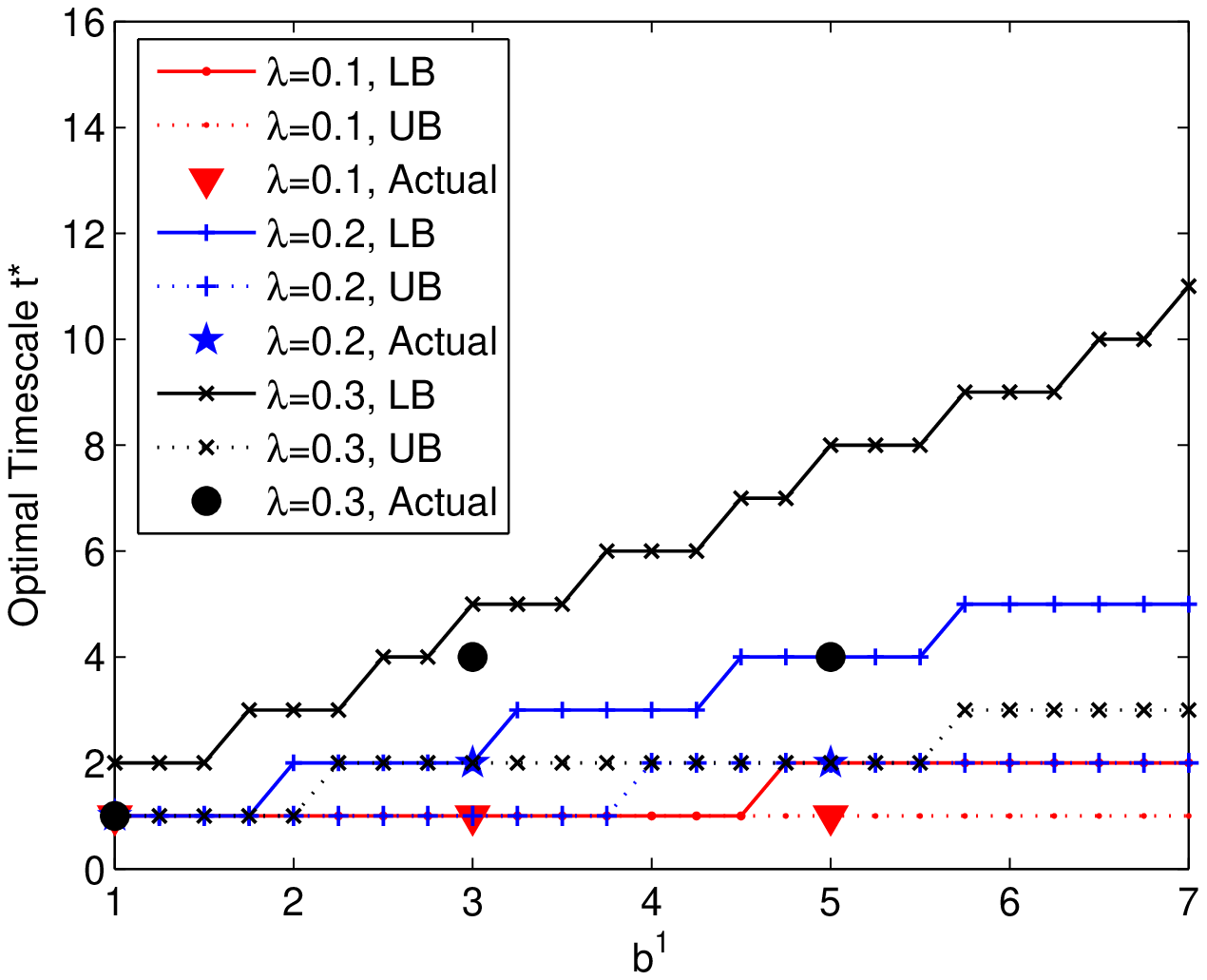} \label{fig:BoundT}}
    \hfill
\caption{Example of the rate function $I_{10}(\bv)$ and its upper and lower bounds, and their corresponding optimizing $t^*$ and optimal trajectories, when $\bv=(b^1,b^2=1)$.} 
\label{fig:BoundCompare} \vspace{-0.1in}
\end{figure}

\begin{figure}[tbp]
\addtocounter{subfigure}{2}  
\centering
    \subfigure[Optimal Trajectory for $\bv=(3,1)$ and $\lambda=0.2$]{
        \includegraphics[width=.75\linewidth]{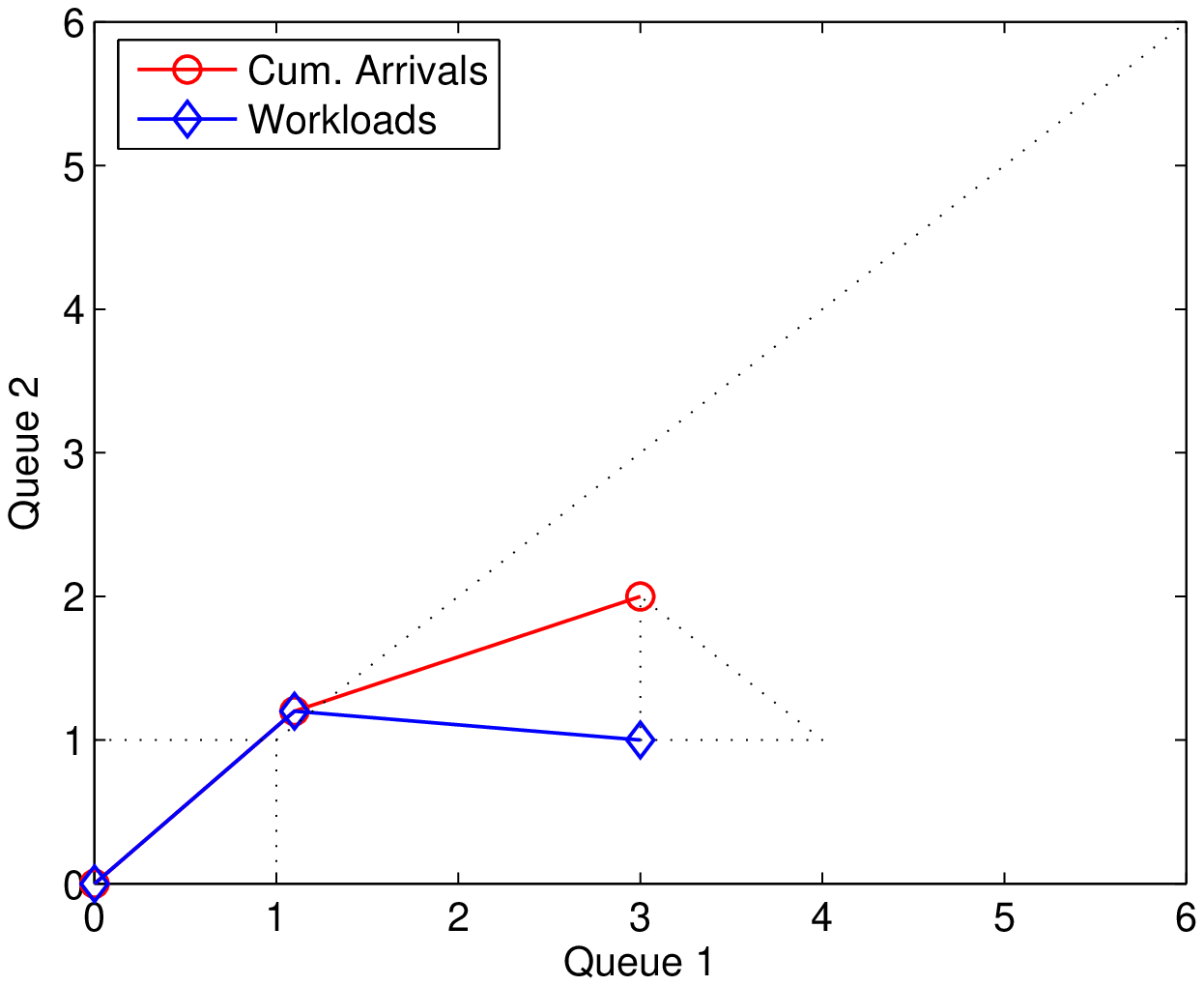} \label{fig:BoundTrajec2}}
    \hfill
    \subfigure[Optimal Trajectory for $\bv=(3,1)$ and $\lambda=0.3$]{
        \includegraphics[width=.75\linewidth]{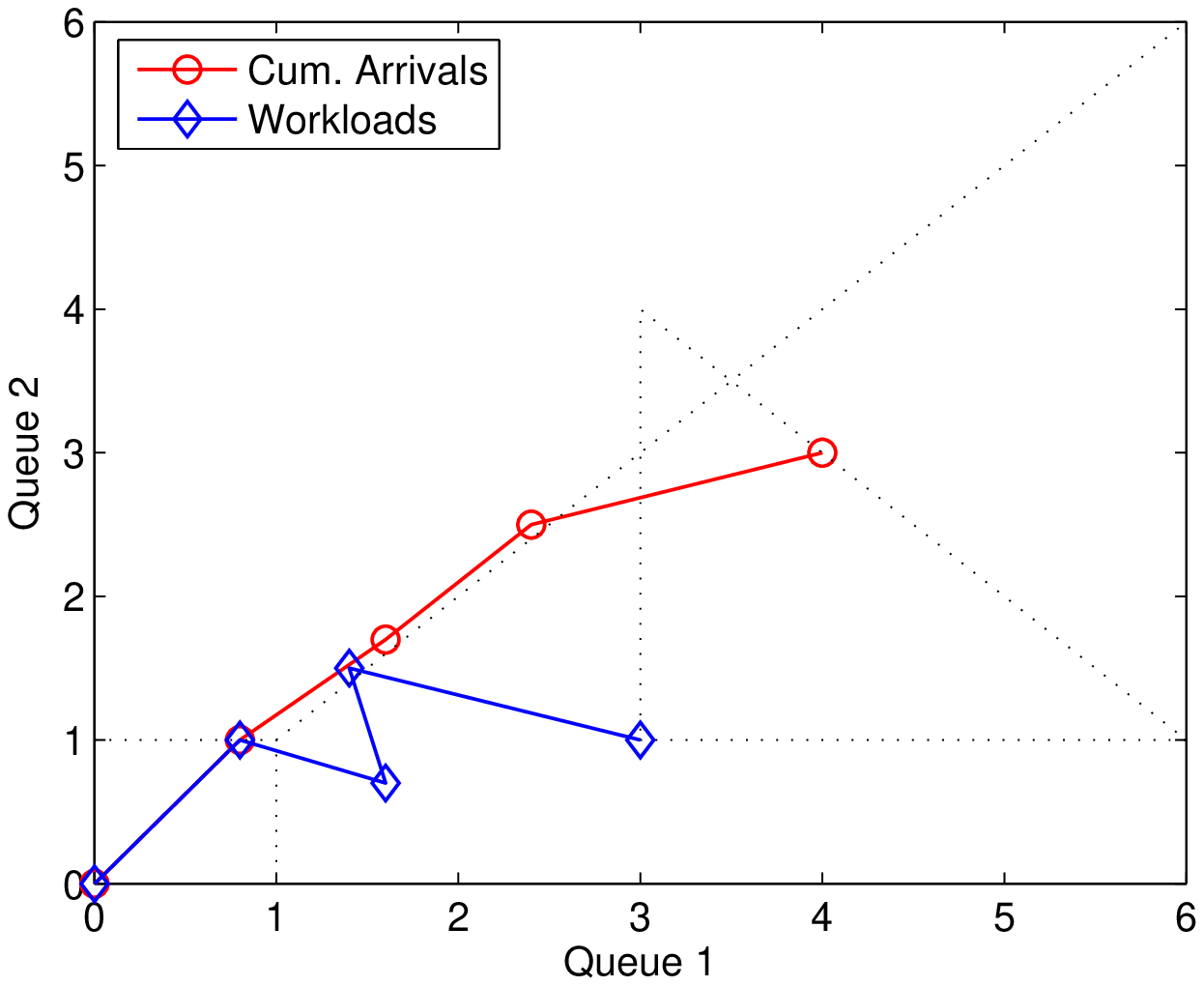} \label{fig:BoundTrajec}}
    \hfill
    \caption{The optimal accumulated arrival and workload trajectories when $\bv=(3,1)$ at $\lambda=0.2$ and $\lambda=0.3$,  respectively.} 
\end{figure} 

Note that despite a potential exponential growth in computation of $I_t$, such 
computations commonly reduce to simpler cases.  For instance, consider the calculation of $I_{10}$ in  Figs.~\ref{fig:BoundTrajec2} and~\ref{fig:BoundTrajec}, in which we sequentially calculate $I_1$, $I_2$, $I_3$. However, this sequential computation stops as soon as one reaches the optimal timescale $t^* < t$. For the values of the vector $b$ in Figs.~\ref{fig:BoundTrajec2} and~\ref{fig:BoundTrajec}, for instance, $t^*$ was at most 4, making the calculations of $I_6$-$I_{10}$ unnecessary.

\section{Concluding Remarks} \label{sect:Conclusion}

In this paper, we have established a many-sources LDP for the stationary (infinite-horizon) workload for multi-queue single-server system with simplex rate region and under maximum-weight scheduling, when the arrival processes assumed to satisfy certain many-sources sample path LDP. To extend the LDP of the arrival processes to the LDP of the workloads, we employed Garcia's extended contraction principle, which applies to quasi-continuous mappings. Along the way, we also establish an LDP for the finite-horizon workload
in a very general setting of arbitrary compact, convex, and coordinate convex
rate region under max-weight scheduling. We gave the associated rate functions and the expression of the infinite-horizon rate function in term of the finite-horizon ones, when the arrival processes have \emph{i.i.d.} increments. 

Next, we catalogue some interesting areas of future research. 
The extension of our LDP result for 
the infinite-horizon workload in  the 
case of an arbitrary compact, convex, and coordinate convex rate region remains 
open by and large. The main difficulty in establishing an LDP for the infinite-horizon workload is in showing the quasi-continuity of the infinite-horizon workload mapping. In the case of simplex rate
region, the infinite-horizon workload mapping was shown to be reducible to a finite-horizon mapping whose quasi-continuity was established via induction. 
Another question that has only been partially explored
in the current paper is the nature of Assumptions~\ref{Assumption-cont}-\ref{Assumption2}. 
In particular, beyond Proposition~1, 
the relationship between stochastic mixing properties of the arrival process
and the analytical properties of $\Ish_t$ and $\Ish$ remains an important area of
future research.



\appendices

\section{Proof of Lemma~\ref{lem:cont_H}}\label{sec:app1}

Lemma~\ref{lem:cont_H} states that it is possible to construct quasi-continuous functions 
$H$ and $H^\mathrm{wc}$ such that 
\beq
H(\Wv_t) \in  \arg\max_{\Rvec \in \RR} <\Rvec , \Wv_t>,  \nonumber
\eeq
and 
\beq 
H^{\mathrm{wc}}(\Wv_t) = \left\{ \begin{array}{ll}
 \text{Proj}_{\RR}(\Wv_t) &  \Wv_t \in \Pi_{k=1}^{K} [0,C_k) \\
H(\Wv_t) &  \mbox{otherwise} \end{array}. \right.  \nonumber
\eeq
We do this using selection theorems (see
\cite{matejdes1987} and \cite{matejdes1991}) for correspondences associated with 
max-weight scheduling:
\begin{eqnarray}\label{eq:maxmono}
{\mathcal{H}(\Wv_t}) & : = & \arg\max_{\Rvec \in \RR} <\Rvec , \Wv_t>.
\end{eqnarray}

First we define the following analytical properties of correspondences.
\begin{defn}[Local-boundedness]\label{defn:lb} A correspondence $F:\XX\rightrightarrows\Y$ is locally-bounded (\emph{lb}) \cite[Defn. 5.15, pp. 157--158]{RockafellarWets1998} at point $x\in\XX$ if there exists a neighbourhood $V$ of $x$ such that $F(V):=\cup_{a\in V}F(a)$ is bounded in $\Y$. Correspondence $F$ is deemed \emph{lb} if it is \emph{lb} for every point $x\in\XX$.
\end{defn}
\begin{defn}[Outer semicontinuity]\label{defn:osc} For a correspondence $F:\XX\rightrightarrows\Y$ define the outer-limit at $x\in\XX$ to be 
\begin{align*}
\limsup_{a\rightarrow x} F(a):&=\cup_{x^n\rightarrow x} \limsup_{n\rightarrow \infty} F(x^n) = \{ y | \exists x^n \rightarrow x,\exists y^n\rightarrow y \text{ with } y^n \in F(x^n)\}.
\end{align*}
Then $F$ is outer semicontinuous (\emph{osc}) \cite[Defn. 5.4, pg. 152]{RockafellarWets1998} at $x\in\XX$ if $\limsup_{a\rightarrow x} F(a) \subset F(x)$. Correspondence $F$ is deemed \emph{osc} if it is \emph{osc} for every point $x\in\XX$.
\end{defn}
\begin{defn}[Upper semicontinuity]\label{defn:usc} A correspondence $F:\XX\rightrightarrows\Y$ is upper semicontinuous (also known as upper hemicontinuous, \cite[Thm. 5.19, pg. 160]{RockafellarWets1998} and \cite{AubinCellina1984}) at point $x\in\XX$ if for any open neighbourhood $V$ of $F(x)$, there exists a neighbourhood $U$ of $x$ such that $F(a)\subseteq V$ for all $a\in U$. Alternatively, $F$ is said to be upper semicontinuous at $x\in\XX$ if whenever we have sequences $\{x_m\}\in \XX$ and $\{ y_m \}\in\Y$ such that $y_m\in F(x_m)$ for all $m$, $x_m \rightarrow x$ and $y_m\rightarrow y$, then $y\in F(x)$. Correspondence $F$ is deemed \emph{usc} if it is \emph{usc} for every point $x\in\XX$.
\end{defn}

Now we can prove Lemma~\ref{lem:cont_H}.
\begin{IEEEproof}[Proof of Lemma~\ref{lem:cont_H}] 
The scheduling function ${\mathcal{H}}(\cdot): \R_+^K \mapsto \mathcal{P}(\RR)$ is a maximal monotone correspondence~\cite[Thm. 12.17, pp. 542--543]{RockafellarWets1998} that picks closed and convex subsets of $\RR$ for every $\xv\in\R_+^K$; thus, compact subsets of $\R_+^K$. It is, therefore, both \emph{lb} and \emph{usc} from \cite[Ex. 12.8b, pg. 536]{RockafellarWets1998}. Now it follows that we get a quasi-continuous selection by \cite[Thm. 2]{matejdes1987} and \cite[Thms. 2.4 \& 2.5]{matejdes1991} since $H(\Wv_t(\cdot))$ is \emph{usc} and has compact values. 

To prove that $H^{\mathrm{wc}}$ is a quasi-continuous selection we first use the same steps as above but with a restricted domain, i.e., for ${\mathcal{H}}(\cdot):\R_+^K\setminus \prod_{k\in\K}[0,C_k) \mapsto \mathcal{P}(\RR)$ since $\R_+^K\setminus \prod_{k\in\K}[0,C_k)$ is a Baire space~\cite{Munkr00} so that results from \cite{matejdes1987,matejdes1991} still apply. From the definition of the work-conserving max-weight scheduler, if $\xv\in \prod_{k\in\K}[0,C_k)$, then we get continuity from within this set from the properties of $\mathrm{Proj}_{\RR}(\cdot)$. Thus, we can satisfy the definition of quasi-continuity from Defn.~\ref{defn:quasi-continuity_condition}.
\end{IEEEproof}
We refer the reader to \cite[Thm. 2.2]{CaoMoors2005} and \cite[Thm. 3.4]{CazacuLawson2007} for an exposition and for generalizations of the result from \cite{matejdes1987,matejdes1991}. 

\section{Proof of Lemmas~\ref{lem:continuity_Gt}-\ref{lem:quasic_mw} }\label{sec:app2}

Next we prove Lemmas~\ref{lem:continuity_Gt} 
which uses the following fact: 

\begin{fact} \label{fact:sum_semi} Assume 
$\XX,\Y$ are metric spaces and $F_1$ and $F_2$ are functions from $\XX$ onto $\Y$. If $F_1$ is quasi-continuous at $x\in \XX$ 
and $F_2$ is continuous at $x$, then $F_1+F_2$ is quasi-continuous at $x$. 
\end{fact}


Below we state and prove Lemma \ref{lem:continuity_Gt} for a simplex rate region $\RR_s$ to 
avoid unnecessary notational complexities; the result readily extends to a general 
rate region as discussed in Remark~\ref{remark7}.

\emph{Lemma \ref{lem:continuity_Gt}:}
For $t\in \N$, $G_t^{\mathrm{wc}}$ for $\RR_s$ is quasi-continuous on $\R^{K t}_+$ with respect to the scaled uniform norm topology.
\begin{IEEEproof} 
Using our queueing equation we first observe the following recursive relation between $G_t^{\mathrm{wc}}$ and $G_{t-1}^{\mathrm{wc}}$ for any $t \in \{2,3,\ldots\}$ and $\xv \in \D^K_\muv$: 
\begin{align} \label{eq:relationGt}
G_{t}^{\mathrm{wc}}(\xv|_{(0,t]})  = \xv_1+  \bigl[G_{t-1}^{\mathrm{wc}}(\xv|_{(1,t]}) 
- H^{\mathrm{wc}}( G_{t-1}^{\mathrm{wc}}(\xv|_{(1,t]}))\bigr]^+  , 
\end{align}
where we used the fact that $\Wv_0^{\mathrm{wc}}(\xv|_{(0,t]}) = G_t^{\mathrm{wc}}(\xv|_{(0,t]})$, and $\Wv_1^{\mathrm{wc}}(\xv|_{(1,t]}) = G_{t-1}^{\mathrm{wc}}(\xv|_{(1,t]})$ when the initial backlog at time $-t$ is $\mathbf{0}$. 

Equation \reqn{eq:relationGt} says that $G_t^{\mathrm{wc}}(\xv|_{(0,t]})$ 
depends linearly on $\xv_1$. This implies the following simple but consequential observations:

\begin{obs} \label{Observation_1}
If $G_t^{\mathrm{wc}}$ is quasi-continuous at $\xv|_{(0,t]}$, then it is quasi-continuous at $\tilde \xv|_{(0,t]} := (\tilde{\xv}_1,\xv_2,\dotsc,\xv_t)$ for any $\tilde{\xv}_1 \in \R^K_+$, and if $G_t^{\mathrm{wc}}$ is continuous at $\xv|_{(0,t]}$, then it is also continuous at $\tilde \xv|_{(0,t]}$.
\end{obs}

\begin{obs}\label{Observation_2}
If $G_t^{\mathrm{wc},\RR_s}(\xv^n|_{(0,t]}) \rightarrow G_t^{\mathrm{wc}}(\xv|_{(0,t]})$ for a sequence $\{\xv^n|_{(0,t]}\}$ such that $\xv^n|_{(0,t]} \rightarrow \xv|_{(0,t]}$, then for any sequence 
$\{\tilde{\xv}^n|_{(0,t]}= (\tilde{\xv}^n_1,\xv^n_2,\dotsc,\xv^n_t)\}$ where
$\tilde{\xv}^n_1 \rightarrow \xv_1$, we also have
$G_t^{\mathrm{wc}}(\tilde{\xv}^n|_{(0,t]}) \rightarrow G_t^{\mathrm{wc}}(\xv|_{(0,t]})$.
\end{obs}

Using the recursive relation in \reqn{eq:relationGt}, we prove this lemma by induction on $t \in \N$. For every $t\in\N$ we assume that the system is empty at $-t$. Therefore, the arrival sample-path prior to $-t$ has no influence on $\Wv_{0,t}^{\mathrm{wc}}$. Thus, we will only specify the values of the constructed sequences up until time $t$; the extension to sample-paths in $\D^K_\muv$ while ensuring that the system is empty at $-t$ is trivial. For $t=1$, $G_1^{\mathrm{wc}}(\av_1) = \av_1$, hence $G_1^{\mathrm{wc}}$ is continuous on $\R^K_+$. Assuming that $G_t^{\mathrm{wc}}$ is quasi-continuous on $\R^{Kt}_+$, we want to show that $G_{t+1}^{\mathrm{wc}}$ is quasi-continuous on $\R^{K(t+1)}_+$. Using the fact that the $[\cdot]^+$ function is continuous, Remark~\ref{Rem:CompositionQuasi}, and  Fact~\ref{fact:sum_semi}, it suffices to show that the function $F_t :=G_t^{\mathrm{wc}}-H^{\mathrm{wc}}\circ G_t^{\mathrm{wc}}$, is quasi-continuous on $\R^{Kt}_+$ to show that $G_{t+1}^{\mathrm{wc}}$ is quasi-continuous. 
In particular, for any arrival sample path $\av  \in \D^K_\muv$, we need to show that $F_t$ is quasi-continuous at $\av|_{(0,t]}$. It suffices to show that it is possible to select a sequence $\hav^n \rightarrow \av$ (in $\D^K_\muv$) for which 
\begin{align}
G_t^{\mathrm{wc}}(\hav^n|_{(0,t]}) & \rightarrow G_t^{\mathrm{wc}}(\av|_{(0,t]}), \label{g} \\
H^{\mathrm{wc}} \circ G_t^{\mathrm{wc}}(\hav^n|_{(0,t]}) & \rightarrow H^{\mathrm{wc}} \circ G_t^{\mathrm{wc}}(\av|_{(0,t]}), \label{hg}
\end{align}
such that both $G_t^{\mathrm{wc}}(\cdot)$ and $H^{\mathrm{wc}}\circ G_t^{\mathrm{wc}}(\cdot)$ are continuous at every $\hav^n|_{(0,t]}$. Note that in contrast to Fact~\ref{fact:sum_semi} we are adding two quasi-continuous and showing that the sum is quasi-continuous; the key to our proof  is to ensure that we use the same sequence for both functions.

We show this by first noting that the induction hypothesis, i.e., quasi-continuity of $G_t^{\mathrm{wc}}$, and the definition of quasi-continuity ensure that there exists a sequence $\{\av^n\}$ such that $\av^n \rightarrow \av$, in the scaled uniform norm topology, such that $G_t^{\mathrm{wc}}(\av^n|_{(0,t]}) \rightarrow G_t^{\mathrm{wc}}(\av|_{(0,t]})$, and $G_t^{\mathrm{wc}}(\cdot)$ is continuous at $\av^n|_{(0,t]}$ for all $n$. We will construct the desired sequence $\{\hav^n\}$ by modifying  $\hat{a}_1^n$ appropriately. 
We proceed by considering the following two cases, depending on the value of $\av_1$.

\textbf{Case 1:} $\av_1>\mathbf{0}$, i.e., every component of the $\av_1\in\R_+^K$ is positive. Let $\epsilon>0$ be the smallest component of $\av_1$, i.e., $\epsilon=\min_{k\in\K} a^k_1$. 
Since $H(\cdot)$ is quasi-continuous, it is possible to choose a 
sequence of (workload) vectors $\{\wv^n\}$ such that
$\wv^n \rightarrow G_t^{\mathrm{wc}}(\av|_{(0,t]}),$ 
and $H$ is continuous at $\wv^n$ for all $n$. 
Now, 
we define 
\begin{align}\label{eq:MQLDP_an}
\begin{split}
\tav^n_1 & := \wv^n - [F_{t-1}(\av^n|_{(1,t]})]^+  = \wv^n - G_t^{\mathrm{wc}}(\av^n|_{(0,t]}) +\av^n_1 \\
& = \left(\wv^n-G_t^{\mathrm{wc}}(\av|_{(0,t]})\right)  +\left(G_t^{\mathrm{wc}}(\av|_{(0,t]})-G_t^{\mathrm{wc}}(\av^n|_{(0,t]})\right) +\left(\av^n_1-\av_1\right)+\av_1. 
\end{split}
\end{align}
It is clear from the last relationship that $\tav_1^n \rightarrow \av_1$. 
We still need to ensure that $\tav^n_1\geq 0$ since negative quantities are involved in the definition. 
We do this by using the facts that  every component of $\av_1\in\R^K$ is greater than or equal to $\epsilon>0$, and that  $\wv^n\rightarrow G_t^{\mathrm{wc}}(\av|_{(0,t]})$, $G_t^{\mathrm{wc}}(\av^n|_{(0,t]})\rightarrow G_t(\av|_{(0,t]})$, and $\av^n_1\rightarrow\av_1$. These facts imply that there exists an $n_{\epsilon}$ such that for all $n> n_{\epsilon}$ we have $\norm{\wv^n-G_t^{\mathrm{wc}}(\av|_{(0,t]})}< \epsilon/3$, $\norm{G_t^{\mathrm{wc},\RR_s}(\av|_{(0,t]})-G_t^{\mathrm{wc}}(\av^n|_{(0,t]})}< \epsilon/3$ and $\norm{\av^n_1-\av_1} < \epsilon/3$ (with the square norm) which then together with \reqn{eq:MQLDP_an} imply that, for the sequence $\tav^{m+n_{\epsilon}}_1$, we always have non-negativity of all components. Hence, we construct a new sequence $\{\hav^{n}_{(0,t]}\}$ where $\hav^{n}_1 = \tav^{n+n_{\epsilon}}_1$ and  $\hav^{n}_{(1,t]} = \av^{n+n_{\epsilon}}_{(1,t]}$.

This new sequence $\hav^{n}_{(0,t]}$ is the sequence we are after because using the induction hypothesis together with Observations~\ref{Observation_1}~and~\ref{Observation_2}, we have that
$G_t^{\mathrm{wc}}(\hav^n|_{(0,t]}) \rightarrow G_t^{\mathrm{wc}}(\av|_{(0,t]}),$ and $G_t^{\mathrm{wc}}$ is continuous at $\hav^n|_{(0,t]}$ for all $n$. Furthermore, by 
construction
\begin{align}
G_t^{\mathrm{wc}}(\hav^n|_{(0,t]}) &= \hav^n_1 + [F_{t-1}(\hav^n|_{(1,t]})]^+  =  \tav^{n+n_{\epsilon}}_1 + [F_{t-1}(\av^{n+n_{\epsilon}}|_{(1,t]})]^+=  \wv^{n+n_{\epsilon}}.
\end{align}
Hence, we have shown that there exists a sequence $\hav^n_{(0,t]}$ satisfying
\reqn{g} and \reqn{hg}. In addition, the continuity of $H^{\mathrm{wc}}\circ G_t^{\mathrm{wc}}$ at $\hav^n|_{(0,t]}$ 
for all $n$ is a direct consequence of continuity of $G_t^{\mathrm{wc}}$ at 
$\hav^n|_{(0,t]}$ and continuity of $H^{\mathrm{wc}}$ at $\wv^{n+n_{\epsilon}}$, which is equal to $G_t^{\mathrm{wc}}(\hav^n|_{(0,t]})$, for all $n$. 

\textbf{Case 2:} $\av_1 \geq \mathbf{0}$. Let $\K_1:=\{k\in \K: a_1^k=0\}$ and let $\K_2:=\arg\max_{k\in \K}  \Wv_0^k(\av|_{(0,t]}) C_k$. Without loss of generality, by permuting the user labels we can assume that the first $\hat{K}:=|\K_1\cup\K_2|$ components of $\av_1$ are either in $\K_1$ or in $\K_2$; thus, the rest of the $K-\hat{K}$ components are both positive and not part of the scheduling decision made at time $0$ with arrival sequence $\av_{(0,t]}$. Now consider the sequence $\av^m_1:=[1/m C]_{\hat{K}} + \av_1$ where $[1/m C]_{\hat{K}}$ is short-hand for a vector with $1/(m C_k)$ in the first $\hat{K}$ components and $0$ in the remaining coefficients; by construction $\av_1^m$ converges to $\av_1$ such that for every $m$ every component of $\av^m_1$ is positive. We construct a sequence $\{\av^m|_{(0,t]}\}$ with this $\av^m_1$ and $\av^m|_{(1,t]} = \av|_{(1,t]}$. It is obvious that $G_t^{\mathrm{wc}}(\av^m|_{(0,t]})\rightarrow G_t^{\mathrm{wc}}(\av|_{(0,t]})$ since 
\begin{align*} 
G_t^{\mathrm{wc}}(\av^m|_{(0,t]}) &= \av^m_1 + [F_{t-1}(\av^m|_{(1,t]})]^+= \av^m_1 + [F_{t-1}(\av|_{(1,t]})]^+  = [1/mC]_{\hat{K}} + G_t^{\mathrm{wc}}(\av|_{(0,t]}).
\end{align*}
When $G_t^{\mathrm{wc},\RR_s}(\av|_{(0,t]}) \not\in [0,C)^K$, by construction, we have 
\begin{align*} 
H^{\mathrm{wc}}\circ G_t^{\mathrm{wc}}(\av^m|_{(0,t]}) &= H^{\mathrm{wc}}\left(G_t^{\mathrm{wc}}(\av^m|_{(0,t]})\right)=H^{\mathrm{wc}}\left(G_t^{\mathrm{wc}}(\av|_{(0,t]})\right)=H\circ G_t^{\mathrm{wc}}(\av|_{(0,t]}),
\end{align*} 
where the function $H(\cdot)$ is the regular max-weight scheduling function (with $\RR_s$). 
On the other hand, if $G_t^{\mathrm{wc}}(\av|_{(0,t]})\in [0,C)^K$, then the continuity of $\text{Proj}_{\RR_s}(\cdot)$ yields $H^{\mathrm{wc}}\circ G_t^{\mathrm{wc}}(\av^m|_{(0,t]})\rightarrow H^{\mathrm{wc}}\circ G_t^{\mathrm{wc}}(\av|_{(0,t]})$. 

Since for each $m$ we have that $\av^m_1$ has all elements strictly positive, we can use the construction from {\bf Case 1} but with $\av^m|_{(0,t]}$ in place of $\av|_{(0,t]}$. 
In particular, for each $m$,
we can now generate a sequence $\{\tav^{m,n}\}$ such that $\tav|_1^{m,n}\rightarrow \av_1^m$ as $n\rightarrow+\infty$, $\tav^{m,n}|_{(1,t]}=\av^n|_{(1,t]}$, and by using Observations~\ref{Observation_1}~and~\ref{Observation_2}, the following hold
\begin{align}
 G_t^{\mathrm{wc}}(\tav^{m,n}|_{(0,t]})  & \rightarrow G_t^{\mathrm{wc}}(\av^m|_{(0,t]}), \label{gm} \\
 H^{\mathrm{wc}} \circ G_t^{\mathrm{wc}}(\tav^{m,n}|_{(0,t]})  & \rightarrow  H^{\mathrm{wc}} \circ G_t^{\mathrm{wc}}(\av^m|_{(0,t]}), 
\label{hgm}
\end{align}
with both $G_t^{\mathrm{wc}}(\cdot)$ and $H^{\mathrm{wc}}\circ G_t^{\mathrm{wc}}(\cdot)$ being continuous at $\tav^{m,n}|_{(0,t]}$ for all $m,\ n$. 

Now we define the sequence $\hav^m=\tav^{m,m}$ as the sequence we are after. By construction, we have $\hav^m|_{(0,t]}\rightarrow \av|_{(0,t]}$ and  both $G_t^{\mathrm{wc}}(\cdot)$ and $H^{\mathrm{wc}} \circ G_t^{\mathrm{wc}}(\cdot)$ continuous at all $\hav^m|_{(0,t]}$. Since $G_t^{\mathrm{wc}}(\av^m|_{(0,t]})\rightarrow G_t^{\mathrm{wc}}(\av|_{(0,t]})$ and $H^{\mathrm{wc}}\circ G_t^{\mathrm{wc}}(\av^m|_{(0,t]})\rightarrow H^{\mathrm{wc}}\circ G_t^{\mathrm{wc}}(\av|_{(0,t]})$, it follows from (\ref{gm}) and (\ref{hgm}) that $G_t^{\mathrm{wc}}(\hav^m|_{(0,t]})\rightarrow G_t^{\mathrm{wc}}(\av|_{(0,t]})$ and $H^{\mathrm{wc},\RR_s}\circ G_t^{\mathrm{wc}}(\hav^m|_{(0,t]})\rightarrow H^{\mathrm{wc}}\circ G_t^{\mathrm{wc}}(\av|_{(0,t]})$. 
\end{IEEEproof}
\begin{note} \label{remark7}
The proof above can be carried out for every rate-region in the class that we are interested in. The argument presented in \textbf{Case 1} would remain exactly the same but the argument presented in \textbf{Case 2} would have to be modified to account for a further characterization of $\arg\max_{\Rvec\in\RR} <\Wv_0(\av|_{(0,t]}), \Rvec>$, especially when it is not a singleton. The components of $\av_1$ will need to be adjusted in such a manner so as to not perturb $\arg\max_{\Rvec\in\RR} <\Wv_0(\av^m|_{(0,t]}), \Rvec>$ for the adjusted sequence $\av^m|_{(0,t]}$. As the case of a non-singleton $\arg\max_{\Rvec\in\RR} <\Wv_0(\av^m|_{(0,t]}), \Rvec>$ will correspond to a specific set of values of $\Wv_0(\av|_{(0,t]})$ (a cone) such that the boundary of the rate-region and a hyper-plane intersect at more than one point, we will need to use the normal corresponding to the hyper-plane in constructing the appropriate perturbation. Thus, on a case-by-case basis the same argument can be carried out for every rate-region. 
\end{note}

Finally, we prove Lemma~\ref{lem:quasic_mw}.

\emph{Lemma \ref{lem:quasic_mw}:} 
For $t\in\N$, $G_t(\cdot)$ is quasi-continuous on $\R^{K\times t}_+$.
\begin{IEEEproof}
Consider the queueing equation, i.e., 
\begin{align*}
\Wv_{T-1}=[\Wv_{T} - \Rvec_T]^+ + \av_T \quad T\in \N,
\end{align*}
where $\Rvec_{(\cdot)}\in \mathcal{H}(\Wv_{(\cdot)})$ where $\mathcal{H}(\cdot)$ is the maximal monotone correspondence defined in (\ref{eq:maxmono}).
Assume that we start the system at (fixed) time $t\in\N$ with workload vector $\Wv_t\in\R_+^K$; we will often assume that $\Wv_t=\mathbf{0}$. First, for $t\in\N$ we define a correspondence $\mathcal{G}_t(\Wv_t,\av|_{(0,t]}): \R_+^{(t+1)K}\rightrightarrows\R_+^K$ that represents all possible workload vectors at time $0$ that can be achieved from the inputs $(\Wv_t,\av|_{(0,t]})$. This results from the successive application of the queueing equation where we use all possible values in $\mathcal{H}(\cdot)$ based upon the workload vector that results at each step. Our goal is show that $\mathcal{G}_{t}(\cdot)$ admits a quasicontinuous selection which we shall call $\tilde{G}_t(\cdot)$. For $t>1$ it suffices to establish that 
\begin{align*}
& \mathcal{G}_{t-1}(\Wv_t,\av|_{(1,t]})-\mathcal{H}(\mathcal{G}_{t-1}(\Wv_t,\av|_{(1,t]})) :=\\
 & \quad \{\xv\in\R^K: \xv=\mathbf{y}-\mathbf{z} \text{ for some } \mathbf{y}\in \mathcal{G}_{t-1}(\Wv_t,\av|_{(1,t]})  \text{and }  \mathbf{z}\in \mathcal{H}(\mathcal{G}_{t-1}(\Wv_t,\av|_{(1,t]})) \}
\end{align*} 
admits a quasicontinuous selection which we shall call (with an abuse of notation) $\tilde{G}_{t-1}(\Wv_t,\av|_{(1,t]})-H(\tilde{G}_{t-1}(\Wv_t,\av|_{(1,t]}))$: 1) since $[\cdot]^+$ is a continuous function, we have \\
$[\tilde{G}_{t-1}(\Wv_t,\av|_{(1,t]} ) - H(\tilde{G}_{t-1}(\Wv_t,\av|_{(1,t]}))]^+$ being a quasi-continuous function; and finally, 2) by properties of projections and by the definition of quasi-continuity we get the quasi-continuity of $\tilde{G}_{t}(\Wv_t,\av|_{(0,t]})$. We should add a note of caution here that even though for notational convenience we write 
\begin{align*}
\tilde{G}_{t}(\Wv_t,\av|_{(0,t]})= [\tilde{G}_{t-1}(\Wv_t,\av|_{(1,t]} )-H(\tilde{G}_{t-1}(\Wv_t,\av|_{(1,t]}))]^++\av_1,
\end{align*} 
it need not be the case that the quasi-continuous selection that we obtain for $\tilde{G}_t$ be related using the above queueing equation to the quasi-continuous selection for $\tilde{G}_{t-1}$. Additionally, we may not even use the quasi-continuous selection $H(\cdot)$. Therefore, the property highlighted in Remark~\ref{rem:rate_calc} need not hold.

The proof will once again use mathematical induction where our induction step will assume that $\mathcal{G}_{t-1}(\cdot)$ is \emph{osc}. Since $\Wv_{0} \leq \Wv_t + \mathbf{a}{(0,t]}$ for any $\Wv_{0}\in\mathcal{G}_t(\Wv_t,\av|_{(0,t]})$, it follows that $\mathcal{G}_t(\Wv_t,\av|_{(0,t]})$ is \emph{lb}. Now using \cite[Ex. 12.8b, pg. 536]{RockafellarWets1998} we know that $\mathcal{H}(\cdot)$ is both \emph{lb} and \emph{osc}, and therefore by \cite[Thm. 5.19, pg. 160]{RockafellarWets1998} it is also \emph{usc}. Then using \cite[Prop. 5.52b, pp. 184--185]{RockafellarWets1998} we have $\mathcal{H}(\mathcal{G}_{t-1}(\cdot))$ being \emph{osc}. Therefore it also follows that $\mathcal{G}_{t-1}(\cdot)-\mathcal{H}(\mathcal{G}_{t-1}(\cdot))$ is also \emph{osc}. Once this has been demonstrated the induction step is very easy as $\mathcal{G}_{t}(\cdot)$ is obtained from $\mathcal{G}_{t-1}(\cdot))-\mathcal{H}(\mathcal{G}_{t-1}(\cdot))$ by continuous transformations, as mentioned above. This same method also allows us to establish the initial step of the induction procedure; note that we will be dealing with $\mathcal{H}(\Wv_t)$ and $\Wv_t-\mathcal{H}(\Wv_t)$ in this case.

Since $\mathcal{G}_{t-1}(\cdot)-\mathcal{H}(\mathcal{G}_{t-1}(\cdot))$ is \emph{lb}, using \cite[Ex. 12.8b, pg. 536]{RockafellarWets1998} we have $\mathcal{G}_{t-1}(\cdot)-\mathcal{H}(\mathcal{G}_{t-1}(\cdot))$ also being \emph{usc}. Finally, we get a quasi-continuous selection by \cite[Thm. 2.2]{CaoMoors2005} and \cite[Thm. 3.4]{CazacuLawson2007} since $\mathcal{G}_{t-1}(\cdot)-\mathcal{H}(\mathcal{G}_{t-1}(\cdot))$ is \emph{usc} and takes compact values. 

The required result then follows by setting $G_t(\av|_{(0,t]})=\tilde{G}_{t}(\mathbf{0},\av|_{(0,t]})$.
\end{IEEEproof}

\section{Proof of Lemmas~\ref{claimzero}-\ref{lem:almost-compact_G}}\label{sec:app3}

First we prove Lemma~\ref{claimzero}:

\emph{Lemma~\ref{claimzero}}: 
Consider an arrival process $\av \in \D^K_\muv$. There exists a $s^*=s^*(\av) < \infty$  such that 
the workloads at time $-s^*$ under $\av$ falls within 
the rate region $\RR_s$, i.e.,
$G^{\mathrm{wc}}(\av|_{(s^*,\infty]}) \in \RR_s$. Furthermore, 
for any sequence of arrival processes $\{\av^n \in \D^K_\muv\}$ converging to
$\av$ (in scaled uniform topology), the workloads at time $-s^*$ under $\av^n$, when
$n$ is large enough, also fall within 
the rate region $\RR_s$, i.e., 
$\exists n_0$ such that $G^{\mathrm{wc}}(\av^n|_{(s^*,\infty]}) \in \RR_s$ for $n> n_0$. 
\begin{IEEEproof}
Consider the normalized sum arrivals and the normalized sum workloads, and follow the proof in \cite{Wisch01,Ganes04} for the (aggregate) single-queue scenario. 
Given the definition of $H^{\mathrm{wc}}$ and the simplex capacity region $\RR_s$, the queue dynamics for the normalized sum workload is that of a single queue whose arrivals are the normalized sum of the arrivals, i.e., 
\beq 
\hat \Wv_{t-1} = [\hat{\Wv}_{t} - 1]^+ + \hat{\av}_{t},\label{eq:sumqdynamic}
\eeq 
where recall that  the hat ($\,\hat{\cdot}\,$) notation means the normalized sum over all users, i.e. $\hat \av_t = \sum_{k=1}^K a^k_t/C^k$ and $\hat \Wv_t = \sum_{k=1}^K W^k_t/C^k$. 
Recursion of the queue dynamics \reqn{eq:sumqdynamic} and letting $T\rightarrow \infty$ where $\Wv_T \in \RR_s$, gives the standard expression for the infinite-horizon sum workload \cite{Ganes04}: 
\beq 
\hat{\Wv}_0 := \hat{G}(\av) =  \sup_{t\in \N} \hat{\av}(0,t]-(t-1),
\eeq 
where $\hat{G}$ represents the infinite-horizon normalized sum workload
mapping; in other words, for all $s$, $\hat{\Wv}_s = \hat{G}(\av|_{(s,\infty]})$ represent the normalized sum workload at time $s$, under arrival sequence $\av$. 

To prove the lemma we use the fact that the rate region $\RR_s$ is simplex, hence $\hat{\Wv}_s \le 1 \Leftrightarrow \Wv_s \in \RR_s$.  Thus, it suffices to show that there is a finite $n_0'$ and a finite $s$ such that $\hat{G}(\av|_{(s,\infty]}) \le 1$, and for $n\ge n_0'$, $\hat{G}(\av^n|_{(s,\infty]}) \le 1$.

Since $\av \in  \D^K_\muv$, there is a $t_0<\infty$ such that for all $\epsilon>0$, 
$t > t_0$ and $k\in K$, $\frac{a^k(0,t]}{t} \le \mu^{k} +\epsilon C^k.$
Since $\muv \in \mbox{int}(\RR_s)$, we choose $\epsilon=(1-\hat{\muv})/4K$. We now have that for all $t\ge t_0$, 
$\frac{\hat{\av}(0,t]}{t} \le \hat{\muv}  + \epsilon K = (1+3\hat{\muv} )/4 < 1.$
In other words,  the workload at time zero is a function of only the arrivals 
within time $(0,t_0]$ and hence, 
\begin{align} \label{Wt1} 
\hat{\Wv}_0(\av) = \sup_{1\le t\le t_0}  \hat{\av}(0,t] - (t-1).
\end{align} 

Let $s^* \le t_0< \infty$ be the minimum values of the optimizing $t$'s in the above equation.
It can be shown \cite[Lemma 5.4]{Ganes04} that 
\[
\hat{G}(\av|_{(s^*,\infty]}) \le 1.
\]

It is known that $\hat{G}$ is continuous on $\D_{\hat{\muv}}$  \cite[Lemma 13]{Wisch01} 
when $\hat{\muv}<1$. 
However,  this together with continuity of shift mapping implies that for all $\{\av^n\}$ such that $\av^n$ converges to $\av$ in scaled uniform topology, 
\begin{equation}\label{letssee}
\hat{G}(\av^n|_{(u,\infty]}) \rightarrow \hat{G}(\av|_{(u,\infty]}) \mbox{      for all  } u\in [0,s^*].
\end{equation}
In particular, (\ref{letssee}) implies that there exists $n_0$ such that for all $n \ge n_0$, the normalized sum workload under arrival sequence $\av^n$ at time $u=s^*$ is no more than 1 packets, i.e., 
\[
\hat{G}(\av^n|_{(s^*,\infty]}) < 1 \text{ for all } n \ge n_0.
\]
However, since the rate region is a simplex,  
the workload vectors at time $s^*$, under $\av$ and $\av^n$ lie in the 
rate region $\RR_s$. Hence, we have the assertion of the lemma.
\end{IEEEproof}

Next, we prove Lemma~\ref{lem:quasi-continuous_G}:

\emph{Lemma \ref{lem:quasi-continuous_G}:}
Let $\muv \in \mbox{int}(\RR_s) $, $\av$ be an arrival sequence with rate $\muv$ with $\Ish(\av)<+\infty$, and
  $\mcW = G^{\mathrm{wc}}(\av)$ be its corresponding steady state 
workload. For any $\mcW  \in \sideset{^\av}{^{\mathrm{wc}}}\letG $ there exists a
sequence of arrivals $\{\av^n \in \D^K_\muv\}$ such that $\av^n$ converges
 to $\av$ in the scaled uniform norm topology, $G^{\mathrm{wc}}(\av^n) \rightarrow G^{\mathrm{wc}}(\av)$, 
 $G^{\mathrm{wc}}$ is continous at $\av^n$, and  $\Ish(\av^n) \rightarrow \Ish(\av)$.
  
\begin{IEEEproof} 
Lemma~\ref{claimzero} implies that 
for any given arrival sequence $\av \in \D^K_\muv$,   there exists a $s^*$ such that 
\[
G^{\mathrm{wc}}(\av) = G_{s^*}^{\mathrm{wc}}(\av|_{(0,s^*]}).
\] 
However, $G_{s^*}^{\mathrm{wc}}$ is quasi-continuous on $\R^{K\times s^*}_+$. 
This implies that there exists a sequence of finite arrivals $\{\hat{\av}^n|_{(0,s^*]}\}$ such that 
\begin{enumerate}
\item $\hat{\av}^n|_{(0,s^*]} \rightarrow \av|_{(0,s^*]}$;  
\item $G_{s^*}^{\mathrm{wc}}$ is continuous  
at $\hat{\av}^n|_{(0,s^*]}$; and
\item $G_{s^*}^{\mathrm{wc}}(\hat{\av}^n|_{(0,s^*]}) \rightarrow 
G_{s^*}^{\mathrm{wc}}(\av|_{(0,s^*]})$. 
\end{enumerate}

Now construct the sequence of arrivals
$\{\av^n\}$ via concatenation of $\av|_{(s^*,\infty]}$ and $\hat{\av}^n|_{(0,s^*]}$. It is immediate that $\av^n\rightarrow \av$.

Appealing to Lemma~\ref{claimzero}, for $n$ large enough (greater than $n_0$) we have
\begin{enumerate}
\item convergence of $G^{\mathrm{wc}}(\av^n)$ to $G^{\mathrm{wc}}(\av)$ since
\begin{align*} 
G^{\mathrm{wc}}(\av^n)
 =  G_{s^*}^{\mathrm{wc}}(\av^n|_{(0,s^*]}) 
\rightarrow  G_{s^*}^{\mathrm{wc}}(\av|_{(0,s^*]}) 
 =  G^{\mathrm{wc}}(\av); 
\end{align*}
\item continuity of $G^{\mathrm{wc}}(\cdot)$ at $\av^n$. For any sequence converging to $\av^n$ in $\D_\muv^K$ by appealing to Lemma~\ref{claimzero} we know that far enough along every sequence only the arrivals in $(0,s^*]$ matter. Now using the fact that projection is continuous on $\D_\muv^K$, we get the result from the continuity of $G_{s^*}^{\mathrm{wc}}$ at $\av^n|_{(0,s^*]}$. 
\end{enumerate}
This establishes the quasi continuity of function $G^{\mathrm{wc}}$. 
Lastly, Assumptions~\ref{Assumption2} and ~\ref{Assumption-cont} ensure  that 
$\Ish(\av^n)\rightarrow \Ish(\av)$.
\end{IEEEproof}

Finally, we prove Lemma~\ref{lem:almost-compact_G}.

\emph{Lemma~\ref{lem:almost-compact_G}:}
If $\muv \in \mbox{int}(\RR_s)$, the mapping 
$G^{\mathrm{wc}}(\cdot)$ is almost compact on $\D^K_{\muv}$ with respect to the scaled uniform norm topology. 
\begin{IEEEproof} 
This follows almost exactly along the same lines as the proof of Lemma \ref{lem:quasi-continuous_G}. For any $\av^n\rightarrow \av$ we proved the existence of a $n_0$ such that for $n\geq n_0$ such that the workload vectors  only depended on the arrivals within time $(0,t_0]$. Thus, the proof of almost compactness simply follows from Lemma \ref{lem:almostcompact}. Note that we have used the fact that the projection operator is continuous on $\D_\muv^K$.
\end{IEEEproof}

\section{Proof of Lemma \ref{lem:MSLDPItBounds}}\label{sec:app4}

Next we prove Lemma \ref{lem:MSLDPItBounds} which gives the bounds on $I_t$.

\emph{Lemma \ref{lem:MSLDPItBounds}:} 
For $K=2$, $\bv \in \R^2_+$, $I_t(\bv)$ can be bounded as 
\begin{align*}
 I_t(\bv) \ge \min_{u\in (0,t]} u\sum_{k=1}^K\Lambda_1^*\left(\frac{1}{u} \Big(\text{Proj}_{\XXX(u,\bv)}(\vect{0})\Big)^k\right) 
\end{align*}
and when $\bv \not\in [0,1)^2$, 
\begin{align*} 
I_t(\bv) \le \min_{u\in (0,t]} u\sum_{k=1}^K\Lambda_1^*\left(\frac{1}{u}(b^k+(u-1)H(\bv)^k)\right),
\end{align*}
where we recall that the convex set $\XXX(u,\bv) \subseteq \R^2_+$ is defined as
\begin{align*} 
\XXX(u,\bv) := & \{\bv+\mathbf{v}:  \mathbf{v}\in\R_+^2\text{ and }v^1+v^2=(u-1)\}.
\end{align*}

\begin{IEEEproof}
Let $\bv \in \R^2_+$, time $u \in (0, t]$, arrival path $\av|_{(0,t]} \in \A(u, \bv)$, and $\Wv_i \in \R^2_+$ be the workload vector at time $-i$ for $i\in (0,u]$. Assume without loss of generality that $t>1$ as it is easy to see that both bounds turn to be $I_1(\bv)$. Owing to this we can also assume that $u>1$ since the terms corresponding to $u=1$ in both bounds evaluate to $I_1(\bv)$.

We first show the lowerbound \reqn{eq:ItLB}. As we have noted earlier that for $\av|_{(0,t]} \in \A(u,\bv)$, the $[\cdot]^+$ function can be removed from the queue dynamics. Hence, we have
$\av(0, u] = \bv + \sum_{i=1}^{u-1} H(\Wv_i)$, where $\Wv_u \in \RR_s$ and $\Wv_i \not\in \RR_s$ for all $i \in (0, u-1]$. Using this and the fact that $H(\Wv_i) \in \{\mathbf{v}\in\R_+^2 : v^1 +v^2 = 1\}$, for all $i \in (0, u-1]$, we have $\av(0, u] \in \XXX(u, \bv)$ where $\XXX(u,\bv)$ is defined above. Now, given any point $\dv\in \XXX(u,\bv)$, the constant-speed linear path with increments of $\dv/u$ is the minimum-cost path among all the paths with the same destination (using Property~1). In addition, among all the paths to destinations in $\XXX(u,\bv)$, the closest constant-speed linear paths $\av^*|_{(0,u]}$ to the equal line is the minimum-cost path (using Property~2). Since the closest point in $\XXX(u,\bv)$ to the equal line is $\text{Proj}_{\XXX(u,\bv)}(\vect{0})$, we have $\av^*|_{(0,u]} = (\av^*_i = \frac{1}{u} \text{Proj}_{\XXX(u,\bv)}(\vect{0}), i \in (0,u])$. 
Since the set of paths with destination in $\XXX(u,\bv)$ includes all paths in $\A(u,\bv)$, from \reqn{eq:defnIt} we have the lowerbound \reqn{eq:ItLB}: 
\begin{align*}
I_t(\bv) &= \min_{u\in (0,t]} \ \inf_{\xv\in \A(u,\bv)} \sum_{k=1}^K \sum_{i=1}^u \Lambda_1^*(x_i^k)\ge \min_{u\in (0,t]} \ \inf_{\xv\in \R^{Kt}_+: \xv(0,u] \in \XXX(u,\bv)} I_t(\xv)\\
			&= \min_{u\in (0,t]} \ \inf_{\xv\in \R^{Ku}_+: \xv \in \XXX(u,\bv)} I_u(\xv)= \min_{u\in (0,t]} \ u\sum_{k=1}^K\Lambda_1^*\left(\frac{1}{u} \Big(\text{Proj}_{\XXX(u,\bv)}(\vect{0})\Big)^k\right). 
\end{align*}

To show the upperbound \reqn{eq:ItUB}, we only need to show that the constant-speed linear path $\av|_{(0,u]} = (\av_i = \frac{1}{u} (\bv+(u-1)H(\bv)), i \in (0,u])$, is in $\A(u,\bv)$, when $\bv \not\in [0,1)^2$. Without loss of generality, we consider only when $b^1\ge b^2$ and $b^1\ge 1$. In this case, we set $H(\bv)=(1,0)$ and the queue dynamics gives 
\begin{align*}
\Wv_i = \frac{(u-i)}{u}(\bv+(u-1)(1,0)) - (u-1-i)(1,0),
\end{align*}
for all $i \in (0,u-1]$. Since $b^1\ge 1$, we have $W_i^1\ge 1$ and $W_i^1 \ge W_i^2$, and hence we can once again set $H(\Wv_i)=(1,0)$ for all $i \in (0,u-1]$. Hence, $\av|_{(0,u]}\in \A(u,\bv)$). 
\end{IEEEproof} 


\end{document}